\documentclass[12pt]{amsart}

\usepackage[english]{babel}
\usepackage[margin=1in]{geometry}
\usepackage{amsmath}
\usepackage{amssymb}
\usepackage{amsthm}
\usepackage{float}
\usepackage{mathtools}
\usepackage{theoremref}
\usepackage{adjustbox}
\usepackage{amsfonts}
\usepackage{graphicx}
\usepackage{xcolor}
\usepackage[hidelinks]{hyperref}
\usepackage{tikz,tikz-cd}
\usepackage[shortlabels]{enumitem}
\usepackage[inkscapelatex=false]{svg}

\newcommand{\Y}{\mathcal{S}_q^Y}
\newcommand{\Sq}{\mathcal{S}_q}

\newcommand{\C}{\mathbb{C}}
\newcommand{\Z}{\mathbb{Z}}
\newcommand{\N}{\mathbb{N}}

\newcommand{\T}{\mathbb{T}}

\author[Higgins]{Vijay Higgins}
\address{Department of Mathematics, University of California, Los Angeles, CA 90095, USA}
\email{higginsv@math.ucla.edu}

\author[Silva]{Bryan Silva}
\address{Department of Mathematics, University of California, Irvine, CA 92697, USA}
\email{bjsilva1@uci.edu}

\author[Somers]{Amy Somers}
\address{Department of Mathematics, University of California, Santa Barbara, CA 93106, USA}
\email{asomers@ucsb.edu}

\author[Stephens]{David Stephens}
\address{Department of Mathematics, University of California, San Diego, CA 92093, USA}
\email{d1stephens@ucsd.edu}

\author[Sundar]{Srijan Sundar}
\address{Mathematical Institute, University of Oxford, Oxford, OX2 6GG, United Kingdom}
\email{srijan.sundar@stcatz.ox.ac.uk}

\author[Wokhlu]{Parth Wokhlu}
\address{Department of Mathematics, University of Chicago, Chicago, IL 60637, USA}
\email{parthw@uchicago.edu}

\title{$\rm{PGL}_2$ webs at roots of unity}


\theoremstyle{definition}
\newtheorem{definition}{Definition}[section]

\newtheorem*{notation}{Notation}

\theoremstyle{definition}

\theoremstyle{definition}

\theoremstyle{plain}

\theoremstyle{definition}

\theoremstyle{plain}
\newtheorem{theorem}[definition]{Theorem}

\theoremstyle{plain}
\newtheorem{lemma}[definition]{Lemma}

\theoremstyle{plain}
\newtheorem{corollary}[definition]{Corollary}

\newtheorem{proposition}[definition]{Proposition}

\theoremstyle{remark}
\newtheorem{remark}[definition]{Remark}

\theoremstyle{definition}
\newtheorem*{definition*}{Definition}

\theoremstyle{definition}
\newtheorem*{notation*}{Notation}

\theoremstyle{plain}
\newtheorem*{theorem*}{Theorem}

\theoremstyle{remark}
\newtheorem*{remark*}{Remark}

\newcommand{\jwnc}{
\begin{tikzpicture}[xscale = 0.15, yscale = 0.2, baseline=-\the\dimexpr\fontdimen22\textfont2\relax]
    \draw[line width = 0.7] (0,-2.5) -- (0,-2);
    \draw[line width = 0.7] (0,-0.5) -- (0,0.5);
    \draw[line width = 0.7] (0, 2.5) -- (0, 2);
    \draw[line width = 0.7] (-1.25,-2) -- (1.25,-2) -- (1.25, -0.5) -- (-1.25, -0.5) -- (-1.25,-2);
    \draw[line width = 0.7] (-1.25,0.5) -- (1.25,0.5) -- (1.25, 2) -- (-1.25, 2) -- (-1.25,0.5);
    \draw (-0.1, -1.25) node [inner sep=0.75pt] {$n$};
    \draw (-0.1, 1.25) node [inner sep=0.75pt] {$c$};
\end{tikzpicture}
}

\newcommand{\cjwn}{
\begin{tikzpicture}[xscale = 0.15, yscale = 0.2, baseline=-\the\dimexpr\fontdimen22\textfont2\relax]
    \draw[line width = 0.7] (0,-2.5) -- (0,-2);
    \draw[line width = 0.7] (0,-0.5) -- (0,0.5);
    \draw[line width = 0.7] (0, 2.5) -- (0, 2);
    \draw[line width = 0.7] (-1.25,-2) -- (1.25,-2) -- (1.25, -0.5) -- (-1.25, -0.5) -- (-1.25,-2);
    \draw[line width = 0.7] (-1.25,0.5) -- (1.25,0.5) -- (1.25, 2) -- (-1.25, 2) -- (-1.25,0.5);
    \draw (-0.1, 1.25) node [inner sep=0.75pt] {$n$};
    \draw (-0.1, -1.25) node [inner sep=0.75pt] {$c$};
\end{tikzpicture}
}

\def\positivecrossing{\tikz[baseline=0]{
		\draw[dashed] (0,0) circle (1);
		\draw[thick] (-.7,-.7) -- (.7,.7);
		\draw[thick] (.1,-.1) -- (.7,-.7);
		\draw[thick] (-.7,.7) -- (-.1,.1);
}}

\def\verticalresolution{\tikz[baseline=0]{
		\draw[dashed] (0,0) circle (1);
		\draw[thick] plot [smooth, tension=1] coordinates {(-.7,.7) (-.3,0) (-.7,-.7)};
		\draw[thick] plot [smooth, tension=1] coordinates {(.7,.7) (.3,0) (.7,-.7)};
}}
\def\horizontalresolution{\tikz[baseline=0]{
		\draw[dashed] (0,0) circle (1);
		\draw[thick] plot [smooth, tension=1] coordinates {(-.7,.7) (0,.3) (.7,.7)};
		\draw[thick] plot [smooth, tension=1] coordinates {(-.7,-.7) (0,-.3) (.7,-.7)};
}}
\def\trivialcomponent{\tikz[baseline=0]{
		\draw[dashed] (0,0) circle (1);
		\draw[thick] (0,0) circle (.5);
}}

\newcommand{\kauffmana}{
\begin{tikzpicture}[xscale = 0.15, yscale = 0.2, baseline=-\the\dimexpr\fontdimen22\textfont2\relax]
    \draw[line width = 0.7] (-1,-1) -- (1,1);
    \draw[line width = 0.7] (-1,1) -- (-0.2, 0.2);
    \draw[line width = 0.7] (1,-1) -- (0.2, -0.2);
\end{tikzpicture}
}

\newcommand{\kauffmanb}{
\begin{tikzpicture}[xscale = 0.15, yscale = 0.2, baseline=-\the\dimexpr\fontdimen22\textfont2\relax]
    \draw[line width = 0.7] [rounded corners] (-1,1) -- (0,0) -- (-1,-1);
    \draw[line width = 0.7] [rounded corners] (1,1) -- (0,0) -- (1,-1);
\end{tikzpicture}
}

\newcommand{\kauffmanc}{
\begin{tikzpicture}[xscale = 0.15, yscale = 0.2, baseline=-\the\dimexpr\fontdimen22\textfont2\relax]
    \draw[line width = 0.7] [rounded corners] (-1,1) -- (0,0) -- (1,1);
    \draw[line width = 0.7] [rounded corners] (-1,-1) -- (0,0) -- (1,-1);
\end{tikzpicture}
}

\newcommand{\jwn}[1]{
\begin{tikzpicture}[xscale = 0.2, yscale = 0.2, baseline=-\the\dimexpr\fontdimen22\textfont2\relax]
    \draw[line width = 0.7] (-1,-1) -- (1,-1) -- (1, 1) -- (-1, 1) -- (-1,-1);
    \draw (0, 0) node [inner sep=0.75pt] {#1};
\end{tikzpicture}
}

\newcommand{\zupper}{
\begin{tikzpicture}[xscale = 0.2, yscale = 0.2, baseline=-\the\dimexpr\fontdimen22\textfont2\relax]
    \draw (0,0) circle (4);
    \draw (0,0) circle (1);
    \draw[line width = 1.2] (-0.4,2.4) arc (100:440:2.5);
    \draw[line width = 1.2] (0, 1) -- (0, 4);
\end{tikzpicture}
}

\newcommand{\zlower}{
\begin{tikzpicture}[xscale = 0.2, yscale = 0.2, baseline=-\the\dimexpr\fontdimen22\textfont2\relax]
    \draw (0,0) circle (4);
    \draw (0,0) circle (1);
    \draw[line width = 1.2] (0,0) circle (2.5);
    \draw[line width = 1.2] (0, 1) -- (0, 2);
    \draw[line width = 1.2] (0, 3) -- (0, 4);
\end{tikzpicture}
}

\begin{document}

\begin{abstract}
We study webs on surfaces for the $\rm{PGL}_2$ skein theory specialized at roots of unity. We construct the $\rm{PGL}_2$ analogue of Chebyshev polynomials and use them to construct central elements in the skein algebra. We then focus on the cases of the 3-punctured sphere and the closed torus. We compute the Frobenius homomorphism for these two surfaces and show that the $\rm{PGL}_2$ Frobenius homomorphism for the torus is not injective when the quantum parameter $q^2$ is a root of unity of even order $n.$
\end{abstract}

\maketitle
\tableofcontents

\section{Introduction}

Skein algebras are objects in quantum topology that are situated at the intersection of algebra and topology. They are closely related to character varieties, quantum groups, cluster algebras, mapping class groups, and TQFTs. The most well-studied skein algebra is the Kauffman bracket skein algebra $\Sq(\Sigma_{g,n})$ of an oriented surface $\Sigma_{g,n}$ of genus $g$ with $n$ punctures. The algebra $\Sq(\Sigma)$ is spanned by framed link diagrams on $\Sigma$ subject to the linear skein relations of equations (\ref{KB crossing relation}) and (\ref{KB loop relation}), which are also the defining relations of the Temperley-Lieb ribbon category. The Kauffman bracket skein algebra may also be viewed as a generalization of the Jones polynomial to links in thickened surfaces.

\begin{align}
\positivecrossing&= q^{1/2} \verticalresolution + q^{-1/2} \horizontalresolution \label{KB crossing relation}\\
\trivialcomponent&=-(q+q^{-1}) \label{KB loop relation}.
\end{align}

The skein algebra is particularly interesting when the quantum parameter $q$ is specialized to a root of unity. In this situation, the algebra exhibits a rich center and admits a Frobenius homomorphism \cite{Bonahon_2015,FKL19,GJS25}.

For higher rank quantum groups, it is an important open problem to compute the Frobenius homomorphism on web basis elements due to the expected connection to canonical bases which exhibit positivity properties and are related to categorified structures \cite{MQ23,Q22}.

The $\rm{PGL}_2$ skein theory, also sometimes referred to as $\rm{SO}_3$ skein theory, arises from a certain subcategory of the Temperley-Lieb category associated to a trivalent vertex whose edges correspond to the second Jones-Wenzl projector. The defining relations appear in (\ref{eq:skein3}). The $\rm{PGL}_2$ skein theory has been studied previously in the literature in the context of applications to spatial graphs; the Yamada, Tutte, and chromatic polynomials; Langlands duality for skein modules; and TQFT \cite{Chb06,Chb12,FK09,FK10,AK16,Jor23,FK23}.

\begin{equation}\label{eq:skein3}
    \begin{matrix}\includegraphics[width=0.7\linewidth]{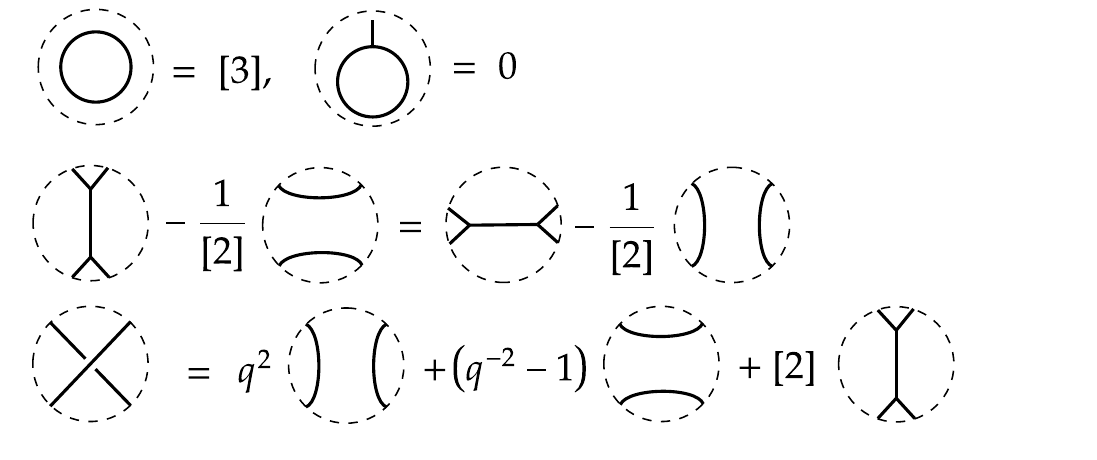}\end{matrix}
\end{equation}

In our setting, we are interested in the $\rm{PGL}_2$ skein algebra of a surface $\Y(\Sigma)$ because it is a setting in which we can study webs on surfaces while still having access to some of the tools available to us from the $\rm{SL}_2$ setting, due to the natural map $\varphi: \Y(\Sigma) \rightarrow \Sq(\Sigma)$ from the $\rm{PGL}_2$ skein algebra to the $\rm{SL}_2$ or Kauffman bracket skein algebra defined by the following assignments:

\begin{equation}
\begin{tikzpicture}[baseline=-15ex,x=0.75pt,y=0.75pt,yscale=-.7,xscale=.7]

\draw    (210,250) -- (210,30) ;
\draw    (60,250) -- (60,30) ;
\draw    (190,250) -- (190,30) ;
\draw  [fill={rgb, 255:red, 255; green, 255; blue, 255 }  ,fill opacity=1 ] (180,120) -- (220,120) -- (220,160) -- (180,160) -- cycle ;
\draw    (80,140) -- (160,140) ;
\draw [shift={(160,140)}, rotate = 180] [color={rgb, 255:red, 0; green, 0; blue, 0 }  ][line width=0.75]    (11,-3) .. controls (7,-1) and (3,0) .. (0,0) .. controls (3,0) and (7,1) .. (11,3)   ;

\draw    (340,90) -- (390,140) ;
\draw    (390,140) -- (440,90) ;
\draw    (390,140) -- (390,200) ;
\draw    (450,140) -- (530,140) ;
\draw [shift={(530,140)}, rotate = 180] [color={rgb, 255:red, 0; green, 0; blue, 0 }  ][line width=0.75]    (10.93,-3.29) .. controls (6.95,-1.4) and (3.31,-0.3) .. (0,0) .. controls (3.31,0.3) and (6.95,1.4) .. (10.93,3.29)   ;
\draw    (550,90) .. controls (580,110) and (620,110) .. (650,90) ;
\draw    (660,100) .. controls (630,120) and (610,160) .. (610,190) ;
\draw    (590,190) .. controls (590,160) and (570,120) .. (540,100) ;
\draw    (80,130) -- (80,150) ;
\draw    (450,130) -- (450,150) ;
\draw  [color={rgb, 255:red, 255; green, 255; blue, 255 }  ,draw opacity=1 ][fill={rgb, 255:red, 255; green, 255; blue, 255 }  ,fill opacity=1 ] (38,7) -- (240,7) -- (240,80) -- (38,80) -- cycle ;
\draw  [color={rgb, 255:red, 255; green, 255; blue, 255 }  ,draw opacity=1 ][fill={rgb, 255:red, 255; green, 255; blue, 255 }  ,fill opacity=1 ] (36,195) -- (238,195) -- (238,268) -- (36,268) -- cycle ;

\draw (120,110) node [anchor=north west][inner sep=0.75pt]    {$\varphi $};
\draw (490,110) node [anchor=north west][inner sep=0.75pt]    {$\varphi $};
\draw (200,140) node [inner sep=0.75pt]   [align=left] {2};
\draw (225, 140) node [anchor=west]{, \quad\ \ and};
\draw (660, 144) node [anchor=west]{.};
\end{tikzpicture}
\end{equation}

In this paper, we study $\rm{PGL}_2$ skein algebras at roots of unity. Due to the denominator appearing in the Jones-Wenzl projector on $2$ strands, we assume that $q+q^{-1}\neq0$ throughout the paper. We examine the centrality of elements that arise from threading the $\rm{PGL}_2$ analogue of the Chebyshev polynomial along knots, and we obtain explicit constructions of the Frobenius homomorphism on the web bases of the $\rm{PGL}_2$ skein algebras of the 3-punctured sphere, $\Sigma_{0,3},$ and the torus, $\Sigma_{1,0}.$

\subsection{Main results.}

In \cite{Bonahon_2015}, it is shown that central elements in the Kauffman bracket skein algebra can be constructed by an operation known as threading (see Section \ref{sec: threading and transparency}), which can be informally described as replacing components of a link by elements represented by Chebyshev polynomials $T_n$ when $q^{2n}=1.$  Here, the polynomial $T_n$ is defined by the property that $T_n(\lambda+\lambda^{-1})=\lambda^n+\lambda^{-n}$ for a formal variable $\lambda.$ In \cite{L_2015} the centrality result was generalized to a notion of transparency for skeins in a 3-manifold. L\^{e} also strengthened the result to show that $T_n$ generates any polynomial that produces transparent elements.

In our first main result, we prove the $\rm{PGL}_2$ analogue of these results, involving a family of polynomials $P_n$ in place of the Chebyshev polynomials $T_n.$ The polynomial $P_n$ is defined by the property that $P_n(\lambda+1+\lambda^{-1})=\lambda^n+1+\lambda^{-n}$ for a formal variable $\lambda.$

\begin{theorem}[Theorem \ref{Thm:transparent}]
If $q^{4n}=1,$ then applying the threading operation from (\ref{eq:threading}) to a link using the polynomial $P_n$ produces a transparent element in the $\rm{PGL}_2$ skein module of any 3-manifold, or a central element in the $\rm{PGL}_2$ skein algebra of any surface.

Furthermore, if a non-constant polynomial $p(x) \in \mathbb{C}[x]$ produces transparent elements via threading, then $q$ must be a root of unity and $p$ must be an element of the polynomial ring $\mathbb{C}[P_N(x)],$ where $N$ is the order of $q^4.$
\end{theorem}

Along the way, we prove several results about versions of the skein algebra of the annulus in Section \ref{sec:annulus}.

The transparent elements we construct are in the image of an expected Frobenius map $\mathcal{S}_1^Y(M) \rightarrow \mathcal{S}_q^Y(M).$ In the context of $\rm{SL}_d$ skein algebras for $d \geq 3,$ it is an important problem to describe the image of the Frobenius map for webs \cite{BH24,KW24,Hig25,KLW25,CHW26}, due to expected connections with bases for skein algebras which would be analogous to the bracelets basis for the Kauffman bracket skein algebra \cite{FG00,MS17,T14,Q22,MQ23}.  Since $\rm{PGL}_2$ skein theory is in some sense the simplest skein theory that requires the use of webs, it is natural to attempt to study the image of the Frobenius homomorphism in this setting. We make progress in this direction by studying the image of the $\rm{PGL}_2$ Frobenius homomorphism on basis webs for both the 3-punctured sphere and the closed torus.

By work of Chbili \cite{Chb12}, the $\rm{PGL}_2$ skein algebra of the 3-punctured sphere, $\Sq^Y(\Sigma_{0,3})$ has a monomial basis generated by $\{a,b,c,d\}$ where the elements $a,b,c$ are represented by boundary loops and the element $d$ is a linear combination of webs defined in Section \ref{Sec: Frobenius pants}.

\begin{theorem}[Theorem \ref{Thm:Frobenius pants}]

When $q^2$ is a root of unity of order $n$ and $t=q^{n^2}$ there is an injective algebra homomorphism

\begin{equation*}
Fr^Y: \mathcal{S}_t^Y(\Sigma_{0,3}) \rightarrow \mathcal{S}_q^Y(\Sigma_{0,3})
\end{equation*}

defined by

\begin{equation*}
Fr^Y(l)=P_n(l) \text{ for } l \in \{a,b,c\}
\end{equation*}

and

\begin{equation*}
Fr^Y(d)=\begin{cases}
d\cdot O_n(a)O_n(b)O_n(c) \text{ if $n$ odd,}\\
E_n(a)E_n(b)E_n(c) \text{ if $n$ even},
\end{cases}
\end{equation*}

Here, $O_n(x)$ is defined to be the unique polynomial satisfying $O_n(x^2-1)=\frac{T_n(x)}{x},$ and $E_n(x)$ is defined to be the unique polynomial satisfying $E_n(x^2-1)=T_n(x).$

\end{theorem}

In the work \cite{ESS_2025} a basis for the $\rm{PGL}_2$ skein algebra of the torus was constructed that is analogous to the Chebyshev basis of Frohman and Gelca \cite{FG00}. The basis is $\{(k,l)_Y, w, \varnothing: (k,l) \in \Z^2\backslash \{(0,0)\}\},$ where $(k,l)_Y$ is a certain polynomial in a torus curve and $w$ is a linear combination of webs given in Section \ref{sec: Frobenius homomorphism for torus}.

\begin{theorem}[Theorem \ref{thm: Frobenius torus}]
When $q^2$ is a root of unity of order $n$ and $t=q^{n^2}$ there is an algebra homomorphism

\begin{equation*}
Fr^Y:\mathcal{S}_t^Y(\Sigma_{1,0})\rightarrow \mathcal{S}_q^Y(\Sigma_{1,0})
\end{equation*}

defined on the basis of $\mathcal{S}_t^Y(\Sigma_{1,0})$ by

\begin{align*}
Fr^Y((k, l)_Y) &= (nk,nl)_Y\\
Fr^Y(w) &= \alpha_n w\\
Fr^Y(\varnothing)&=\varnothing,
\end{align*}

where

\begin{align*}
\alpha_n = \begin{dcases}
0, & n \text{ even} \\ 
\frac{4}{[2]_{q}^2}, & n \text{ odd}.
\end{dcases}
\end{align*}
\end{theorem}
The Frobenius homomorphism can be used to characterize the center of the skein algebra of the torus, which was carried out in \cite{ESS_2025}.

By the description of the $\rm{PGL}_2$ Frobenius homomorphism for the torus we obtain the following corollary.

\begin{corollary}\label{cor: non-injective}
The $\rm{PGL}_2$ Frobenius homomorphism for the torus $Fr^Y: \mathcal{S}_t^Y(\Sigma_{1,0}) \rightarrow \mathcal{S}_q^Y(\Sigma_{1,0})$ is not injective if $n:=\text{ord}(q^2)$ is even but is injective if $n$ is odd.
\end{corollary}

We note that the non-injectivity of $Fr^Y$ for the case of the torus when $n$ is even is in contrast to the injectivity of the Frobenius homomorphism for the Kauffman bracket skein algebra \cite{Bonahon_2015}. Although in \cite{CL25} it was discovered that the Frobenius homomorphism for skein modules of 3-manifolds is sometimes non-injective, Corollary \ref{cor: non-injective} contains the first instance to our knowledge of non-injectivity for a Frobenius homomorphism in the case of a surface.

\section{Preliminaries}

Throughout the paper we will assume that $q^{1/2} \in \mathbb{C}$ represents a complex number such that $q+q^{-1} \neq 0.$ Although the paper will focus on $\rm{PGL}_2$ skein modules $\Y(M)$, we will use some tools from the theory of the Kauffman bracket skein module $\Sq(M)$. 

\begin{definition}
The Kauffman bracket skein module $\Sq(M)$ of an oriented 3-manifold $M$ is the $\mathbb{C}$-vector space spanned by isotopy classes of framed links in $M$ modulo the Kauffman bracket skein relations:

\begin{align}
\positivecrossing&= q^{1/2} \verticalresolution + q^{-1/2} \horizontalresolution\\
\trivialcomponent&=-(q+q^{-1}).
\end{align}

When $M=\Sigma \times (0,1)$ is a thickened surface, $\Sq(M)$ is endowed with an algebra structure and we denote the skein algebra by $\Sq(\Sigma).$ The product of two links $L_1 \cdot L_2 \in \Sq(\Sigma)$ is given by isotoping $L_1$ to obtain $L_1' \subset \Sigma \times (1/2,1),$ isotoping $L_2$ to obtain $L_2' \subset  \Sigma \times (0,1/2)$ and then taking $L_1' \cup L_2' \subset \Sigma\times(0,1).$ The skein algebra is an associative, unital algebra over $\mathbb{C}$ with the empty link serving as the multiplicative unit.
\end{definition}

When working with skein algebras of thickened surfaces, we often work with link diagrams on the surface.

\begin{definition}
A $\rm{PGL}_2$ web in an oriented $3$-manifold $M$ is an embedded trivalent graph in $M$ equipped with a thin oriented surface which deformation retracts onto the web. We do not require the graph to be connected, and we permit the empty graph. In the case the graph consists of vertexless loops, the associated web will be the same as a framed link in $M.$
\end{definition}

\begin{definition}
The $\rm{PGL}_2$ skein module $\Y(M)$ of an oriented 3-manifold $M$ is the $\mathbb{C}$-vector space spanned by isotopy classes of $\rm{PGL}_2$ webs in $M$ modulo the $\rm{PGL}_2$ skein relations:

\begin{equation}\label{I-H}
    \begin{matrix}\includegraphics[width=0.7\linewidth]{Threading-2/SqY_Relations.pdf}\end{matrix}
\end{equation}

Here, $[2]_q=q+q^{-1}$ and $[3]_q=q^2+1+q^{-2}$ denote quantum integers. We will often leave off the subscript of $[k]_q$ when the context is clear.

When $M=\Sigma \times (0,1)$ is a thickened surface, $\Y(M)$ is endowed with an algebra structure and we denote the $\rm{PGL}_2$ skein algebra of $\Sigma$ by $\Y(\Sigma).$ The product of two webs $W_1 \cdot W_2 \in \Y(\Sigma)$ is given by isotoping $W_1$ to obtain $W_1' \subset \Sigma \times (1/2,1),$ isotoping $W_2$ to obtain $W_2' \subset \Sigma \times (0,1/2)$ and then taking $W_1' \cup W_2' \subset \Sigma\times(0,1).$ The skein algebra is an associative, unital algebra over $\mathbb{C}$ with the empty web serving as the multiplicative unit.
\end{definition}

When computing with $\rm{PGL}_2$ skein algebras of surfaces we will often work with web diagrams on the surface. We will sometimes refer to the skein relation from the second line of (\ref{I-H}) as an ``I-H move."

We recall there exists a homomorphism $\varphi: \Y(M) \rightarrow \Sq(M)$ from the $\rm{PGL}_2$ skein module to the Kauffman bracket skein module. The map $\varphi$ is defined on a web by the following local replacements,

\begin{equation}\label{eq:phimap}
\begin{tikzpicture}[baseline=-15ex,x=0.75pt,y=0.75pt,yscale=-.7,xscale=.7]

\draw    (210,250) -- (210,30) ;
\draw    (60,250) -- (60,30) ;
\draw    (190,250) -- (190,30) ;
\draw  [fill={rgb, 255:red, 255; green, 255; blue, 255 }  ,fill opacity=1 ] (180,120) -- (220,120) -- (220,160) -- (180,160) -- cycle ;
\draw    (80,140) -- (160,140) ;
\draw [shift={(160,140)}, rotate = 180] [color={rgb, 255:red, 0; green, 0; blue, 0 }  ][line width=0.75]    (11,-3) .. controls (7,-1) and (3,0) .. (0,0) .. controls (3,0) and (7,1) .. (11,3)   ;

\draw    (340,90) -- (390,140) ;
\draw    (390,140) -- (440,90) ;
\draw    (390,140) -- (390,200) ;
\draw    (450,140) -- (530,140) ;
\draw [shift={(530,140)}, rotate = 180] [color={rgb, 255:red, 0; green, 0; blue, 0 }  ][line width=0.75]    (10.93,-3.29) .. controls (6.95,-1.4) and (3.31,-0.3) .. (0,0) .. controls (3.31,0.3) and (6.95,1.4) .. (10.93,3.29)   ;
\draw    (550,90) .. controls (580,110) and (620,110) .. (650,90) ;
\draw    (660,100) .. controls (630,120) and (610,160) .. (610,190) ;
\draw    (590,190) .. controls (590,160) and (570,120) .. (540,100) ;
\draw    (80,130) -- (80,150) ;
\draw    (450,130) -- (450,150) ;
\draw  [color={rgb, 255:red, 255; green, 255; blue, 255 }  ,draw opacity=1 ][fill={rgb, 255:red, 255; green, 255; blue, 255 }  ,fill opacity=1 ] (38,7) -- (240,7) -- (240,80) -- (38,80) -- cycle ;
\draw  [color={rgb, 255:red, 255; green, 255; blue, 255 }  ,draw opacity=1 ][fill={rgb, 255:red, 255; green, 255; blue, 255 }  ,fill opacity=1 ] (36,195) -- (238,195) -- (238,268) -- (36,268) -- cycle ;

\draw (120,110) node [anchor=north west][inner sep=0.75pt]    {$\varphi $};
\draw (490,110) node [anchor=north west][inner sep=0.75pt]    {$\varphi $};
\draw (200,140) node [inner sep=0.75pt]   [align=left] {2};
\draw (225, 140) node [anchor=west]{, \quad\ \ and};
\draw (660, 144) node [anchor=west]{,};
\end{tikzpicture}
\end{equation}

where the box labeled by $2$ corresponds to the second Jones-Wenzl idempotent in the Temperley-Lieb algebra, which is equal to the following expression:

\begin{equation*}
\begin{tikzpicture}[
  line width=0.8pt,
  baseline={([yshift=-0.5ex]current bounding box.center)}
]
  \draw (0,0) -- (0,0.3)
        (0.5,0) -- (0.5,0.3);
  \draw (0,0.7) -- (0,1)
        (0.5,0.7) -- (0.5,1);
  \draw (-0.15,0.3) rectangle (0.65,0.7);
  \node at (0.25,0.5) {$2$};
\end{tikzpicture}
\;=\;
\begin{tikzpicture}[
  line width=0.8pt,
  baseline={([yshift=-0.5ex]current bounding box.center)}
]
  \draw (0,0) -- (0,1);
  \draw (0.5,0) -- (0.5,1);
\end{tikzpicture}
\;+\;\frac{1}{q+q^{-1}}\;
\begin{tikzpicture}[
  line width=0.8pt,
  baseline={([yshift=-0.5ex]current bounding box.center)}
]
  \draw (0,1) .. controls (0,0.5) and (0.5,0.5) .. (0.5,1);
  \draw (0,0) .. controls (0,0.5) and (0.5,0.5) .. (0.5,0);
\end{tikzpicture}.
\end{equation*}

The methods of \cite{Chb12} show that $\varphi$ is injective in the cases that $M$ is a thickened annulus or thickened 3-punctured sphere. However, $\varphi$ is not always injective. For example, in \cite{ESS_2025} it was shown $\varphi$ is not injective in the case of the thickened closed torus.

\section{Tangles and Webs on the Annulus}\label{sec:annulus}

This section will explore tangled webs on the annulus in preparation for constructing central elements in the skein module in the next section. We recall definitions from the Kauffman bracket setting.

\begin{definition}
An $(n,m)$ tangle on the annulus is defined to be a tangle on the annulus with $n$ marked points on the outer boundary and $m$ marked points on the inner boundary. The $\C$-vector space spanned by such tangles subject to the Kauffman bracket skein relations is denoted $A_{n,m}.$ 
\end{definition}

In general, $A_{n,m}$ is only a vector space over $\C$. In the special case that $n = m$, we can endow it with a multiplicative structure. Two diagrams $D_1,D_2 \in A_{n,n}$ are multiplied by placing $D_1$ around the outside of $D_2$ and concatenating.

\begin{notation}
We will be using the following notation for elements of $A_{n,n}$. The identity element of $A_{n,n}$ will be denoted by $e$ and corresponds to the diagram consisting of $n$ parallel straight strands connecting the two marked points. The elements $u, u^{-1} \in A_{n,n}$ are defined to be a single twist of all strands around the annular hole. Note that these elements are in the center of $A_{n,n}$.
\begin{figure}[h!]
    \centering
    \resizebox{9cm}{!}{

\begin{tikzpicture}[x=0.75pt,y=0.75pt,yscale=-1,xscale=1]

\draw   (410.53,121.02) .. controls (410.53,72.13) and (370.9,32.5) .. (322.02,32.5) .. controls (273.13,32.5) and (233.5,72.13) .. (233.5,121.02) .. controls (233.5,169.9) and (273.13,209.53) .. (322.02,209.53) .. controls (370.9,209.53) and (410.53,169.9) .. (410.53,121.02) -- cycle ;
\draw   (345.12,121.02) .. controls (345.12,108.25) and (334.78,97.91) .. (322.02,97.91) .. controls (309.25,97.91) and (298.91,108.25) .. (298.91,121.02) .. controls (298.91,133.78) and (309.25,144.12) .. (322.02,144.12) .. controls (334.78,144.12) and (345.12,133.78) .. (345.12,121.02) -- cycle ;
\draw [line width=2.25]    (322.02,97.91) .. controls (322.27,67.75) and (300.74,73.11) .. (291.88,79.68) ;
\draw [line width=2.25]    (322.02,32.5) .. controls (322.13,48.56) and (320.03,67.68) .. (341.73,73.99) ;
\draw  [draw opacity=0][line width=2.25]  (294.39,78.02) .. controls (279.67,86.98) and (269.89,102.89) .. (269.89,121.02) .. controls (269.89,149.02) and (293.23,171.72) .. (322.02,171.72) .. controls (350.81,171.72) and (374.15,149.02) .. (374.15,121.02) .. controls (374.15,99.43) and (360.27,80.99) .. (340.72,73.68) -- (322.02,121.02) -- cycle ; \draw  [line width=2.25]  (294.39,78.02) .. controls (279.67,86.98) and (269.89,102.89) .. (269.89,121.02) .. controls (269.89,149.02) and (293.23,171.72) .. (322.02,171.72) .. controls (350.81,171.72) and (374.15,149.02) .. (374.15,121.02) .. controls (374.15,99.43) and (360.27,80.99) .. (340.72,73.68) ;  
\draw   (452.83,121.45) .. controls (452.83,72.57) and (492.46,32.94) .. (541.35,32.94) .. controls (590.24,32.94) and (629.87,72.57) .. (629.87,121.45) .. controls (629.87,170.34) and (590.24,209.97) .. (541.35,209.97) .. controls (492.46,209.97) and (452.83,170.34) .. (452.83,121.45) -- cycle ;
\draw   (518.24,121.45) .. controls (518.24,108.69) and (528.59,98.35) .. (541.35,98.35) .. controls (554.11,98.35) and (564.46,108.69) .. (564.46,121.45) .. controls (564.46,134.22) and (554.11,144.56) .. (541.35,144.56) .. controls (528.59,144.56) and (518.24,134.22) .. (518.24,121.45) -- cycle ;
\draw [line width=2.25]    (541.35,98.35) .. controls (541.09,68.19) and (562.63,73.55) .. (571.49,80.12) ;
\draw [line width=2.25]    (541.35,32.94) .. controls (541.23,49) and (543.33,68.12) .. (521.64,74.43) ;
\draw  [draw opacity=0][line width=2.25]  (568.98,78.45) .. controls (583.69,87.41) and (593.48,103.32) .. (593.48,121.45) .. controls (593.48,149.45) and (570.14,172.15) .. (541.35,172.15) .. controls (512.56,172.15) and (489.22,149.45) .. (489.22,121.45) .. controls (489.22,99.86) and (503.09,81.43) .. (522.65,74.12) -- (541.35,121.45) -- cycle ; \draw  [line width=2.25]  (568.98,78.45) .. controls (583.69,87.41) and (593.48,103.32) .. (593.48,121.45) .. controls (593.48,149.45) and (570.14,172.15) .. (541.35,172.15) .. controls (512.56,172.15) and (489.22,149.45) .. (489.22,121.45) .. controls (489.22,99.86) and (503.09,81.43) .. (522.65,74.12) ;  
\draw   (191.53,121.02) .. controls (191.53,72.13) and (151.9,32.5) .. (103.02,32.5) .. controls (54.13,32.5) and (14.5,72.13) .. (14.5,121.02) .. controls (14.5,169.9) and (54.13,209.53) .. (103.02,209.53) .. controls (151.9,209.53) and (191.53,169.9) .. (191.53,121.02) -- cycle ;
\draw   (126.12,121.02) .. controls (126.12,108.25) and (115.78,97.91) .. (103.02,97.91) .. controls (90.25,97.91) and (79.91,108.25) .. (79.91,121.02) .. controls (79.91,133.78) and (90.25,144.12) .. (103.02,144.12) .. controls (115.78,144.12) and (126.12,133.78) .. (126.12,121.02) -- cycle ;
\draw [line width=2.25]    (103.02,32.5) -- (103.02,97.91) ;
\draw  [draw opacity=0][fill={rgb, 255:red, 255; green, 255; blue, 255 }  ,fill opacity=1 ] (91.71,57.39) -- (112.57,57.39) -- (112.57,79.29) -- (91.71,79.29) -- cycle ;

\draw (310.33,220.35) node [anchor=north west][inner sep=0.75pt]  [font=\huge]  {$u$};
\draw  [color={rgb, 255:red, 255; green, 255; blue, 255 }  ,draw opacity=1 ][fill={rgb, 255:red, 255; green, 255; blue, 255 }  ,fill opacity=1 ]  (310.69,151.23) -- (331.69,151.23) -- (331.69,184.23) -- (310.69,184.23) -- cycle  ;
\draw (313.69,165.63) node [anchor=north west][inner sep=0.75pt]  [font=\Large]  {$n$};
\draw (525,210.12) node [anchor=north west][inner sep=0.75pt]  [font=\huge]  {$u^{-1}$};
\draw  [color={rgb, 255:red, 255; green, 255; blue, 255 }  ,draw opacity=1 ][fill={rgb, 255:red, 255; green, 255; blue, 255 }  ,fill opacity=1 ]  (530.82,151.27) -- (551.82,151.27) -- (551.82,184.27) -- (530.82,184.27) -- cycle  ;
\draw (533.82,165.67) node [anchor=north west][inner sep=0.75pt]  [font=\Large]  {$n$};
\draw (94.33,220.35) node [anchor=north west][inner sep=0.75pt]  [font=\huge]  {$e$};
\draw  [draw opacity=0][line width=0.75]   (92.73,48.92) -- (113.73,48.92) -- (113.73,81.92) -- (92.73,81.92) -- cycle  ;
\draw (95.73,63.32) node [anchor=north west][inner sep=0.75pt]  [font=\Large]  {$n$};

\end{tikzpicture}

}
\end{figure}\noindent
As the notation suggests, $u u^{-1} = e$. 

We also define $u^{1/n}, u^{-1/n} \in A_{n,n}$ as pictured below (where $m = n-1$). Observe that $(u^{1/n})^n = u$, and similarly for $u^{-1/n}$.

\begin{figure}[h!]
    \centering
    \resizebox{6cm}{!}{

\begin{tikzpicture}[x=0.75pt,y=0.75pt,yscale=-1,xscale=1]

\draw   (230.53,121.02) .. controls (230.53,72.13) and (190.9,32.5) .. (142.02,32.5) .. controls (93.13,32.5) and (53.5,72.13) .. (53.5,121.02) .. controls (53.5,169.9) and (93.13,209.53) .. (142.02,209.53) .. controls (190.9,209.53) and (230.53,169.9) .. (230.53,121.02) -- cycle ;
\draw   (165.12,121.02) .. controls (165.12,108.25) and (154.78,97.91) .. (142.02,97.91) .. controls (129.25,97.91) and (118.91,108.25) .. (118.91,121.02) .. controls (118.91,133.78) and (129.25,144.12) .. (142.02,144.12) .. controls (154.78,144.12) and (165.12,133.78) .. (165.12,121.02) -- cycle ;
\draw [line width=1.5]    (137.42,98.44) .. controls (134.09,69.55) and (120.74,73.11) .. (111.88,79.68) ;
\draw [line width=1.5]    (144.98,33.1) .. controls (145.09,49.17) and (143.87,67.55) .. (161.73,73.99) ;
\draw  [draw opacity=0][line width=1.5]  (114.39,78.02) .. controls (99.67,86.98) and (89.89,102.89) .. (89.89,121.02) .. controls (89.89,149.02) and (113.23,171.72) .. (142.02,171.72) .. controls (170.81,171.72) and (194.15,149.02) .. (194.15,121.02) .. controls (194.15,99.43) and (180.27,80.99) .. (160.72,73.68) -- (142.02,121.02) -- cycle ; \draw  [line width=1.5]  (114.39,78.02) .. controls (99.67,86.98) and (89.89,102.89) .. (89.89,121.02) .. controls (89.89,149.02) and (113.23,171.72) .. (142.02,171.72) .. controls (170.81,171.72) and (194.15,149.02) .. (194.15,121.02) .. controls (194.15,99.43) and (180.27,80.99) .. (160.72,73.68) ;  
\draw [line width=2.25]    (146.09,97.99) .. controls (147.64,83.1) and (138.29,72.54) .. (135.87,62.21) .. controls (133.45,51.89) and (134.31,38.88) .. (134.31,33.1) ;
\draw  [draw opacity=0][fill={rgb, 255:red, 255; green, 255; blue, 255 }  ,fill opacity=1 ] (126.69,49.5) -- (142.4,49.5) -- (142.4,68.82) -- (126.69,68.82) -- cycle ;
\draw   (273.5,121.62) .. controls (273.5,72.73) and (313.13,33.1) .. (362.02,33.1) .. controls (410.9,33.1) and (450.53,72.73) .. (450.53,121.62) .. controls (450.53,170.51) and (410.9,210.14) .. (362.02,210.14) .. controls (313.13,210.14) and (273.5,170.51) .. (273.5,121.62) -- cycle ;
\draw   (338.91,121.62) .. controls (338.91,108.86) and (349.25,98.51) .. (362.02,98.51) .. controls (374.78,98.51) and (385.12,108.86) .. (385.12,121.62) .. controls (385.12,134.38) and (374.78,144.73) .. (362.02,144.73) .. controls (349.25,144.73) and (338.91,134.38) .. (338.91,121.62) -- cycle ;
\draw [line width=1.5]    (366.61,99.04) .. controls (369.94,70.15) and (383.3,73.72) .. (392.15,80.29) ;
\draw [line width=1.5]    (359.06,33.71) .. controls (358.94,49.77) and (360.17,68.15) .. (342.31,74.6) ;
\draw  [draw opacity=0][line width=1.5]  (389.64,78.62) .. controls (404.36,87.58) and (414.15,103.49) .. (414.15,121.62) .. controls (414.15,149.62) and (390.81,172.32) .. (362.02,172.32) .. controls (333.23,172.32) and (309.89,149.62) .. (309.89,121.62) .. controls (309.89,100.03) and (323.76,81.59) .. (343.31,74.28) -- (362.02,121.62) -- cycle ; \draw  [line width=1.5]  (389.64,78.62) .. controls (404.36,87.58) and (414.15,103.49) .. (414.15,121.62) .. controls (414.15,149.62) and (390.81,172.32) .. (362.02,172.32) .. controls (333.23,172.32) and (309.89,149.62) .. (309.89,121.62) .. controls (309.89,100.03) and (323.76,81.59) .. (343.31,74.28) ;  
\draw [line width=2.25]    (357.94,98.6) .. controls (356.39,83.71) and (365.75,73.14) .. (368.17,62.82) .. controls (370.59,52.49) and (369.72,39.48) .. (369.72,33.71) ;
\draw  [draw opacity=0][fill={rgb, 255:red, 255; green, 255; blue, 255 }  ,fill opacity=1 ] (362.19,49.61) -- (377.9,49.61) -- (377.9,68.93) -- (362.19,68.93) -- cycle ;

\draw (126.83,214.85) node [anchor=north west][inner sep=0.75pt]  [font=\huge]  {$u^{1/n}$};
\draw  [color={rgb, 255:red, 255; green, 255; blue, 255 }  ,draw opacity=1 ][fill={rgb, 255:red, 255; green, 255; blue, 255 }  ,fill opacity=1 ]  (142.77, 169.68) circle [x radius= 11.31, y radius= 20.51]   ;
\draw (135.77,160.58) node [anchor=north west][inner sep=0.75pt]  [font=\Large]  {$1$};
\draw  [draw opacity=0]  (116.25,39.78) -- (149.25,39.78) -- (149.25,72.78) -- (116.25,72.78) -- cycle  ;
\draw (115.25,54.18) node [anchor=north west][inner sep=0.75pt]  [font=\Large]  {$\ m$};
\draw (342.83,215.45) node [anchor=north west][inner sep=0.75pt]  [font=\huge]  {$u^{-1/n}$};
\draw  [color={rgb, 255:red, 255; green, 255; blue, 255 }  ,draw opacity=1 ][fill={rgb, 255:red, 255; green, 255; blue, 255 }  ,fill opacity=1 ]  (362.77, 170.29) circle [x radius= 11.31, y radius= 20.51]   ;
\draw (356.77,161.19) node [anchor=north west][inner sep=0.75pt]  [font=\Large]  {$1$};
\draw  [draw opacity=0]  (358.11,40.38) -- (385.11,40.38) -- (385.11,73.38) -- (358.11,73.38) -- cycle  ;
\draw (361.11,54.78) node [anchor=north west][inner sep=0.75pt]  [font=\Large]  {$m$};

\end{tikzpicture}

}
\end{figure}

Lastly, we will denote the element corresponding to the core loop of the annulus by $z.$ The element $z^n \in A_{0,0}$ then corresponds to $n$ copies of the core loop.

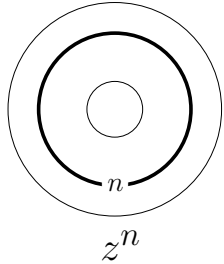
\begin{figure}[H]
    \centering
    \resizebox{3cm}{!}{

\begin{tikzpicture}[x=0.75pt,y=0.75pt,yscale=-1,xscale=1]

\draw   (41.5,117.02) .. controls (41.5,68.13) and (81.13,28.5) .. (130.02,28.5) .. controls (178.9,28.5) and (218.53,68.13) .. (218.53,117.02) .. controls (218.53,165.9) and (178.9,205.53) .. (130.02,205.53) .. controls (81.13,205.53) and (41.5,165.9) .. (41.5,117.02) -- cycle ;
\draw   (106.91,117.02) .. controls (106.91,104.25) and (117.25,93.91) .. (130.02,93.91) .. controls (142.78,93.91) and (153.12,104.25) .. (153.12,117.02) .. controls (153.12,129.78) and (142.78,140.12) .. (130.02,140.12) .. controls (117.25,140.12) and (106.91,129.78) .. (106.91,117.02) -- cycle ;
\draw  [line width=2.25]  (66.87,117.02) .. controls (66.87,82.14) and (95.14,53.87) .. (130.02,53.87) .. controls (164.89,53.87) and (193.16,82.14) .. (193.16,117.02) .. controls (193.16,151.89) and (164.89,180.16) .. (130.02,180.16) .. controls (95.14,180.16) and (66.87,151.89) .. (66.87,117.02) -- cycle ;
\draw (117,216.35) node [anchor=north west][inner sep=0.75pt]  [font=\huge]  {$z^{n}$};
\draw  [color={rgb, 255:red, 255; green, 255; blue, 255 }  ,draw opacity=1 ][fill={rgb, 255:red, 255; green, 255; blue, 255 }  ,fill opacity=1 ]  (119.69,159.23) -- (140.69,159.23) -- (140.69,192.23) -- (119.69,192.23) -- cycle  ;
\draw (122.69,173.63) node [anchor=north west][inner sep=0.75pt]  [font=\Large]  {$n$};

\end{tikzpicture}

}
    \caption{A diagram with $n$ nested core loops.}
\end{figure}
\end{notation}

We begin by recalling the following results:
\begin{proposition}[\cite{Green_1998}]\label{A_n_n}
Let $\beta_n$ be the standard basis of crossingless matchings for $TL_n$ for $n \geq 1$. Then $A_{n,n}$ is spanned by \begin{align*}
&\left\{u^{k/n}bu^{j/n}:
b\in\beta_n,\ k\in\mathbb Z,\ 0\le j<n\right\}\\
{}\cup{}&
\left\{u^{k/n}b[z^r]u^{j/n}:
b\in\beta_n^0,\ r\ge1,\ k\in\mathbb Z,\ 0\le j<n\right\}.
\end{align*}

Here $\beta_n^0$ consists of the matchings with no through strands, and $b[z^r]$ denotes the insertion of $r$ disjoint core loops into $b$.
\end{proposition}

\begin{proposition}[\cite{PS_2000,SW07}]\label{linear independence}
A basis of $A_{k,n}$ is given by isotopy classes of crossingless diagrams in $A_{k,n}$ with no trivial loops.
\end{proposition}

Using Proposition \ref{linear independence} we may construct an explicit list of basis elements for $A_{n,n}$ by taking one element in each isotopy class from the explicit spanning set provided by Proposition \ref{A_n_n}. The two most important bases we will be using are for $A_{1,1}$ and $A_{2,2}$.

By applying the above strategy, a basis for $A_{1,1}$ is given by:
\begin{equation}\label{eq:A_1_1}
\{u^k: k \in \Z\}.
\end{equation}
A basis for $A_{2,2}$ is given by:
\begin{equation}\label{eq:A_2_2}
\{u^{k/2}, f_1z^j, f_2z^j, f_3z^j, f_4z^j: k \in \Z, j \in \N_0\}.
\end{equation}
Here $f_1, f_2, f_3$, and $f_4$ are defined as in Figure \ref{fig:A_2_2}.

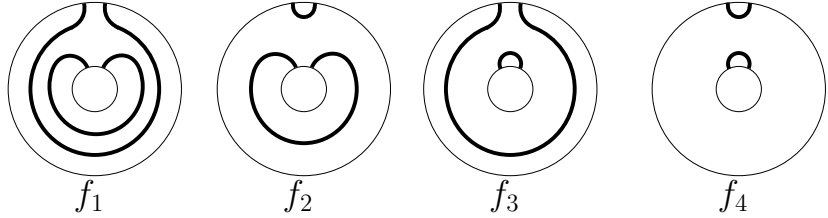
\begin{figure}[h!]
\centering
\resizebox{11cm}{!}{

\begin{tikzpicture}[x=0.75pt,y=0.75pt,yscale=-1,xscale=1]

\draw   (541.5,122.48) .. controls (541.5,85.62) and (571.38,55.75) .. (608.23,55.75) .. controls (645.09,55.75) and (674.96,85.62) .. (674.96,122.48) .. controls (674.96,159.33) and (645.09,189.21) .. (608.23,189.21) .. controls (571.38,189.21) and (541.5,159.33) .. (541.5,122.48) -- cycle ;
\draw   (590.81,122.48) .. controls (590.81,112.86) and (598.61,105.06) .. (608.23,105.06) .. controls (617.85,105.06) and (625.65,112.86) .. (625.65,122.48) .. controls (625.65,132.1) and (617.85,139.9) .. (608.23,139.9) .. controls (598.61,139.9) and (590.81,132.1) .. (590.81,122.48) -- cycle ;
\draw  [draw opacity=0][line width=2.25]  (599.95,59.47) .. controls (600.32,63.72) and (603.89,67.06) .. (608.23,67.06) .. controls (612.58,67.06) and (616.14,63.72) .. (616.51,59.47) -- (608.23,58.75) -- cycle ; \draw  [line width=2.25]  (599.95,59.47) .. controls (600.32,63.72) and (603.89,67.06) .. (608.23,67.06) .. controls (612.58,67.06) and (616.14,63.72) .. (616.51,59.47) ;  
\draw   (45.54,122.48) .. controls (45.54,85.62) and (75.41,55.75) .. (112.27,55.75) .. controls (149.12,55.75) and (179,85.62) .. (179,122.48) .. controls (179,159.33) and (149.12,189.21) .. (112.27,189.21) .. controls (75.41,189.21) and (45.54,159.33) .. (45.54,122.48) -- cycle ;
\draw   (94.85,122.48) .. controls (94.85,112.86) and (102.65,105.06) .. (112.27,105.06) .. controls (121.89,105.06) and (129.69,112.86) .. (129.69,122.48) .. controls (129.69,132.1) and (121.89,139.9) .. (112.27,139.9) .. controls (102.65,139.9) and (94.85,132.1) .. (94.85,122.48) -- cycle ;
\draw  [draw opacity=0][line width=2.25]  (129.58,74.9) .. controls (148.43,82.09) and (161.85,100.68) .. (161.85,122.48) .. controls (161.85,150.51) and (139.65,173.24) .. (112.27,173.24) .. controls (84.88,173.24) and (62.69,150.51) .. (62.69,122.48) .. controls (62.69,100.68) and (76.11,82.09) .. (94.95,74.9) -- (112.27,122.48) -- cycle ; \draw  [line width=2.25]  (129.58,74.9) .. controls (148.43,82.09) and (161.85,100.68) .. (161.85,122.48) .. controls (161.85,150.51) and (139.65,173.24) .. (112.27,173.24) .. controls (84.88,173.24) and (62.69,150.51) .. (62.69,122.48) .. controls (62.69,100.68) and (76.11,82.09) .. (94.95,74.9) ;  
\draw [line width=2.25]    (93.64,75.51) .. controls (97.41,73.94) and (106.16,68.39) .. (104,56.57) ;
\draw [line width=2.25]    (130.66,75.36) .. controls (126.89,73.79) and (118.14,68.24) .. (120.3,56.42) ;
\draw  [draw opacity=0][line width=2.25]  (600.83,106.84) .. controls (600.25,105.71) and (599.92,104.42) .. (599.92,103.06) .. controls (599.92,98.47) and (603.64,94.75) .. (608.23,94.75) .. controls (612.82,94.75) and (616.54,98.47) .. (616.54,103.06) .. controls (616.54,104.27) and (616.28,105.42) .. (615.82,106.46) -- (608.23,103.06) -- cycle ; \draw  [line width=2.25]  (600.83,106.84) .. controls (600.25,105.71) and (599.92,104.42) .. (599.92,103.06) .. controls (599.92,98.47) and (603.64,94.75) .. (608.23,94.75) .. controls (612.82,94.75) and (616.54,98.47) .. (616.54,103.06) .. controls (616.54,104.27) and (616.28,105.42) .. (615.82,106.46) ;  
\draw [line width=2.25]    (600,60.16) -- (599.71,55.88) ;
\draw [line width=2.25]    (616.46,60) -- (616.96,56.3) ;
\draw   (365.54,122.48) .. controls (365.54,85.62) and (395.41,55.75) .. (432.27,55.75) .. controls (469.12,55.75) and (499,85.62) .. (499,122.48) .. controls (499,159.33) and (469.12,189.21) .. (432.27,189.21) .. controls (395.41,189.21) and (365.54,159.33) .. (365.54,122.48) -- cycle ;
\draw   (414.85,122.48) .. controls (414.85,112.86) and (422.65,105.06) .. (432.27,105.06) .. controls (441.89,105.06) and (449.69,112.86) .. (449.69,122.48) .. controls (449.69,132.1) and (441.89,139.9) .. (432.27,139.9) .. controls (422.65,139.9) and (414.85,132.1) .. (414.85,122.48) -- cycle ;
\draw  [draw opacity=0][line width=2.25]  (424.87,106.84) .. controls (424.29,105.71) and (423.96,104.42) .. (423.96,103.06) .. controls (423.96,98.47) and (427.68,94.75) .. (432.27,94.75) .. controls (436.86,94.75) and (440.58,98.47) .. (440.58,103.06) .. controls (440.58,104.27) and (440.32,105.42) .. (439.85,106.46) -- (432.27,103.06) -- cycle ; \draw  [line width=2.25]  (424.87,106.84) .. controls (424.29,105.71) and (423.96,104.42) .. (423.96,103.06) .. controls (423.96,98.47) and (427.68,94.75) .. (432.27,94.75) .. controls (436.86,94.75) and (440.58,98.47) .. (440.58,103.06) .. controls (440.58,104.27) and (440.32,105.42) .. (439.85,106.46) ;  
\draw  [draw opacity=0][line width=2.25]  (449.58,74.9) .. controls (468.43,82.09) and (481.85,100.68) .. (481.85,122.48) .. controls (481.85,150.51) and (459.65,173.24) .. (432.27,173.24) .. controls (404.88,173.24) and (382.69,150.51) .. (382.69,122.48) .. controls (382.69,100.68) and (396.11,82.09) .. (414.95,74.9) -- (432.27,122.48) -- cycle ; \draw  [line width=2.25]  (449.58,74.9) .. controls (468.43,82.09) and (481.85,100.68) .. (481.85,122.48) .. controls (481.85,150.51) and (459.65,173.24) .. (432.27,173.24) .. controls (404.88,173.24) and (382.69,150.51) .. (382.69,122.48) .. controls (382.69,100.68) and (396.11,82.09) .. (414.95,74.9) ;  
\draw [line width=2.25]    (413.64,75.51) .. controls (417.41,73.94) and (426.16,68.39) .. (424,56.57) ;
\draw [line width=2.25]    (450.66,75.36) .. controls (446.89,73.79) and (438.14,68.24) .. (440.3,56.42) ;
\draw   (206.54,122.48) .. controls (206.54,85.62) and (236.41,55.75) .. (273.27,55.75) .. controls (310.12,55.75) and (340,85.62) .. (340,122.48) .. controls (340,159.33) and (310.12,189.21) .. (273.27,189.21) .. controls (236.41,189.21) and (206.54,159.33) .. (206.54,122.48) -- cycle ;
\draw   (255.85,122.48) .. controls (255.85,112.86) and (263.65,105.06) .. (273.27,105.06) .. controls (282.89,105.06) and (290.69,112.86) .. (290.69,122.48) .. controls (290.69,132.1) and (282.89,139.9) .. (273.27,139.9) .. controls (263.65,139.9) and (255.85,132.1) .. (255.85,122.48) -- cycle ;
\draw  [draw opacity=0][line width=2.25]  (264.99,59.47) .. controls (265.36,63.72) and (268.92,67.06) .. (273.27,67.06) .. controls (277.61,67.06) and (281.18,63.72) .. (281.55,59.47) -- (273.27,58.75) -- cycle ; \draw  [line width=2.25]  (264.99,59.47) .. controls (265.36,63.72) and (268.92,67.06) .. (273.27,67.06) .. controls (277.61,67.06) and (281.18,63.72) .. (281.55,59.47) ;  
\draw [line width=2.25]    (265.04,60.16) -- (264.75,55.88) ;
\draw [line width=2.25]    (281.5,60) -- (282,56.3) ;
\draw  [draw opacity=0][line width=2.25]  (311.66,108.51) .. controls (313.17,112.87) and (314,117.58) .. (314,122.48) .. controls (314,145.51) and (295.76,164.18) .. (273.27,164.18) .. controls (250.77,164.18) and (232.54,145.51) .. (232.54,122.48) .. controls (232.54,117.58) and (233.36,112.87) .. (234.88,108.51) -- (273.27,122.48) -- cycle ; \draw  [line width=2.25]  (311.66,108.51) .. controls (313.17,112.87) and (314,117.58) .. (314,122.48) .. controls (314,145.51) and (295.76,164.18) .. (273.27,164.18) .. controls (250.77,164.18) and (232.54,145.51) .. (232.54,122.48) .. controls (232.54,117.58) and (233.36,112.87) .. (234.88,108.51) ;  
\draw [line width=2.25]    (311.87,109.08) .. controls (305.24,88.72) and (284.19,93.6) .. (279.39,105.8) ;
\draw [line width=2.25]    (234.65,109.08) .. controls (241.24,89.45) and (262.35,93.6) .. (267.15,105.8) ;
\draw  [draw opacity=0][line width=2.25]  (147.13,108.15) .. controls (148.46,112.01) and (149.19,116.16) .. (149.19,120.48) .. controls (149.19,140.79) and (133.11,157.26) .. (113.27,157.26) .. controls (93.43,157.26) and (77.34,140.79) .. (77.34,120.48) .. controls (77.34,116.16) and (78.07,112.01) .. (79.41,108.15) -- (113.27,120.48) -- cycle ; \draw  [line width=2.25]  (147.13,108.15) .. controls (148.46,112.01) and (149.19,116.16) .. (149.19,120.48) .. controls (149.19,140.79) and (133.11,157.26) .. (113.27,157.26) .. controls (93.43,157.26) and (77.34,140.79) .. (77.34,120.48) .. controls (77.34,116.16) and (78.07,112.01) .. (79.41,108.15) ;  
\draw [line width=2.25]    (147.32,108.66) .. controls (141.46,90.7) and (123,95.3) .. (118.77,106.06) ;
\draw [line width=2.25]    (79.21,108.66) .. controls (85.02,91.35) and (102.37,95.3) .. (106.6,106.06) ;

\draw (590.71,191.51) node [anchor=north west][inner sep=0.75pt]  [font=\LARGE]  {$f_{4}$};
\draw (94.75,191.51) node [anchor=north west][inner sep=0.75pt]  [font=\LARGE]  {$f_{1}$};
\draw (414.75,191.51) node [anchor=north west][inner sep=0.75pt]  [font=\LARGE]  {$f_{3}$};
\draw (255.75,191.51) node [anchor=north west][inner sep=0.75pt]  [font=\LARGE]  {$f_{2}$};

\end{tikzpicture}

}
\caption{Basis elements in $A_{2,2}$}
\label{fig:A_2_2}
\end{figure}

We also define $v, u' \in A_{2,0}$ as in the figure below:
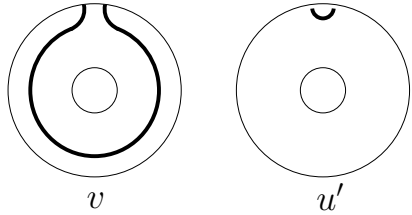
\begin{figure}[h!]
\centering
\resizebox{5.5cm}{!}{

\begin{tikzpicture}[x=0.75pt,y=0.75pt,yscale=-1,xscale=1]

\draw   (541.5,122.48) .. controls (541.5,85.62) and (571.38,55.75) .. (608.23,55.75) .. controls (645.09,55.75) and (674.96,85.62) .. (674.96,122.48) .. controls (674.96,159.33) and (645.09,189.21) .. (608.23,189.21) .. controls (571.38,189.21) and (541.5,159.33) .. (541.5,122.48) -- cycle ;
\draw   (590.81,122.48) .. controls (590.81,112.86) and (598.61,105.06) .. (608.23,105.06) .. controls (617.85,105.06) and (625.65,112.86) .. (625.65,122.48) .. controls (625.65,132.1) and (617.85,139.9) .. (608.23,139.9) .. controls (598.61,139.9) and (590.81,132.1) .. (590.81,122.48) -- cycle ;
\draw  [draw opacity=0][line width=2.25]  (599.95,59.47) .. controls (600.32,63.72) and (603.89,67.06) .. (608.23,67.06) .. controls (612.58,67.06) and (616.14,63.72) .. (616.51,59.47) -- (608.23,58.75) -- cycle ; \draw  [line width=2.25]  (599.95,59.47) .. controls (600.32,63.72) and (603.89,67.06) .. (608.23,67.06) .. controls (612.58,67.06) and (616.14,63.72) .. (616.51,59.47) ;  

\draw   (365.54,122.48) .. controls (365.54,85.62) and (395.41,55.75) .. (432.27,55.75) .. controls (469.12,55.75) and (499,85.62) .. (499,122.48) .. controls (499,159.33) and (469.12,189.21) .. (432.27,189.21) .. controls (395.41,189.21) and (365.54,159.33) .. (365.54,122.48) -- cycle ;
\draw   (414.85,122.48) .. controls (414.85,112.86) and (422.65,105.06) .. (432.27,105.06) .. controls (441.89,105.06) and (449.69,112.86) .. (449.69,122.48) .. controls (449.69,132.1) and (441.89,139.9) .. (432.27,139.9) .. controls (422.65,139.9) and (414.85,132.1) .. (414.85,122.48) -- cycle ;
\draw  [draw opacity=0][line width=2.25]  (449.58,74.9) .. controls (468.43,82.09) and (481.85,100.68) .. (481.85,122.48) .. controls (481.85,150.51) and (459.65,173.24) .. (432.27,173.24) .. controls (404.88,173.24) and (382.69,150.51) .. (382.69,122.48) .. controls (382.69,100.68) and (396.11,82.09) .. (414.95,74.9) -- (432.27,122.48) -- cycle ; \draw  [line width=2.25]  (449.58,74.9) .. controls (468.43,82.09) and (481.85,100.68) .. (481.85,122.48) .. controls (481.85,150.51) and (459.65,173.24) .. (432.27,173.24) .. controls (404.88,173.24) and (382.69,150.51) .. (382.69,122.48) .. controls (382.69,100.68) and (396.11,82.09) .. (414.95,74.9) ;  
\draw [line width=2.25]    (413.64,75.51) .. controls (417.41,73.94) and (426.16,68.39) .. (424,56.57) ;
\draw [line width=2.25]    (450.66,75.36) .. controls (446.89,73.79) and (438.14,68.24) .. (440.3,56.42) ;

\draw (602,193) node [anchor=north west][inner sep=0.75pt]  [font=\LARGE]  {$u'$};
\draw (425,200) node [anchor=north west][inner sep=0.75pt]  [font=\LARGE]  {$v$};

\end{tikzpicture}

}
\caption{Basis elements of $A_{2,0}$}
\label{fig:A_2_0}
\end{figure}

\noindent

Proposition \ref{linear independence} implies that the elements 
\begin{equation}\label{eq:A_2_0}
    \{z^ku', z^lv : k,l \in \N_0\}
\end{equation}
constitute a basis for $A_{2,0}$, which we shall employ in the proof of Lemma \ref{A_1_0_lemma}.

\bigskip

We next describe the $PGL_2$ analogue of $A_{n,m},$ which we denote by $A_{n,m}^Y.$

\begin{definition}
An $(n,m)$ web diagram on the annulus is a web with $n$ univalent endpoints on the outer boundary and $m$ univalent endpoints on the inner boundary. The $\C$ vector space spanned by $(n,m)$ webs subject to the $PGL_2$ relations in (\ref{eq:skein3}) is denoted $A_{n,m}^Y$.
\end{definition}

Similar to the Kauffman bracket version, in the case $m=n,$ we define multiplication in $A_{n,n}^Y$ to be given by concatenation of annular diagrams.

For brevity, we will use the same notation as $A_{n,n}$ for webs with no vertices in $A_{n,n}^Y$. In particular, powers of $u$ and $z$ will represent the same diagrams as in $A_{n,n}$.

\begin{lemma}\thlabel{A_1_0_lemma}
$\{\omega'z^j : j \in \N_0\}$ is a basis for $A_{1,0}^Y$ where $\omega'$ and $z$ are defined as below.
\end{lemma}

\begin{figure}[h!]
\centering
\resizebox{6cm}{!}{

\begin{tikzpicture}[x=0.75pt,y=0.75pt,yscale=-1,xscale=1]

\draw   (45.5,124.02) .. controls (45.5,75.13) and (85.13,35.5) .. (134.02,35.5) .. controls (182.9,35.5) and (222.53,75.13) .. (222.53,124.02) .. controls (222.53,172.9) and (182.9,212.53) .. (134.02,212.53) .. controls (85.13,212.53) and (45.5,172.9) .. (45.5,124.02) -- cycle ;
\draw   (110.91,124.02) .. controls (110.91,111.25) and (121.25,100.91) .. (134.02,100.91) .. controls (146.78,100.91) and (157.12,111.25) .. (157.12,124.02) .. controls (157.12,136.78) and (146.78,147.12) .. (134.02,147.12) .. controls (121.25,147.12) and (110.91,136.78) .. (110.91,124.02) -- cycle ;
\draw  [line width=2.25]  (70.87,124.02) .. controls (70.87,89.14) and (99.14,60.87) .. (134.02,60.87) .. controls (168.89,60.87) and (197.16,89.14) .. (197.16,124.02) .. controls (197.16,158.89) and (168.89,187.16) .. (134.02,187.16) .. controls (99.14,187.16) and (70.87,158.89) .. (70.87,124.02) -- cycle ;
\draw [line width=2.25]    (134.02,35.5) -- (134.02,60.87) ;
\draw   (265.5,124.02) .. controls (265.5,75.13) and (305.13,35.5) .. (354.02,35.5) .. controls (402.9,35.5) and (442.53,75.13) .. (442.53,124.02) .. controls (442.53,172.9) and (402.9,212.53) .. (354.02,212.53) .. controls (305.13,212.53) and (265.5,172.9) .. (265.5,124.02) -- cycle ;
\draw   (330.91,124.02) .. controls (330.91,111.25) and (341.25,100.91) .. (354.02,100.91) .. controls (366.78,100.91) and (377.12,111.25) .. (377.12,124.02) .. controls (377.12,136.78) and (366.78,147.12) .. (354.02,147.12) .. controls (341.25,147.12) and (330.91,136.78) .. (330.91,124.02) -- cycle ;
\draw  [line width=2.25]  (290.87,124.02) .. controls (290.87,89.14) and (319.14,60.87) .. (354.02,60.87) .. controls (388.89,60.87) and (417.16,89.14) .. (417.16,124.02) .. controls (417.16,158.89) and (388.89,187.16) .. (354.02,187.16) .. controls (319.14,187.16) and (290.87,158.89) .. (290.87,124.02) -- cycle ;

\draw (120.5,220.35) node [anchor=north west][inner sep=0.75pt] [font=\huge]   {$\omega' $};
\draw (343,232.45) node [anchor=north west][inner sep=0.75pt] [font=\huge]   {$z$};

\end{tikzpicture}

}
\end{figure}

\begin{proof}
We first prove that $\{\omega'z^j : j \in \N_0\}$ spans $A_{1, 0}^Y$. Suppose $w \in A_{1, 0}^Y$ is a $\mathbb{C}$-linear combination of tangle diagrams. For each tangle diagram, we eliminate crossings via \eqref{eq:skein3}. The result is a linear combination of crossingless disjoint unions of trivalent graphs. For each such disjoint union of trivalent graphs, there is exactly one connected component that contains the marked point. Any other connected component is a closed web in the annulus and may be reduced by the skein relations \eqref{eq:skein3} to a polynomial in $z$, by the arguments of  \cite{Chb06}.

It remains to prove that the component containing the marked point may be expressed as a linear combination of elements of the form $\{\omega'z^j : j \in \N_0\}$. We proceed by induction on the number of vertices in this component. When $n=1$, this component is isotopic to either the web pictured:

\begin{figure}[h!]
\resizebox{3cm}{!}{

\begin{tikzpicture}[x=0.75pt,y=0.75pt,yscale=-1,xscale=1]

\draw   (45.5,124.02) .. controls (45.5,75.13) and (85.13,35.5) .. (134.02,35.5) .. controls (182.9,35.5) and (222.53,75.13) .. (222.53,124.02) .. controls (222.53,172.9) and (182.9,212.53) .. (134.02,212.53) .. controls (85.13,212.53) and (45.5,172.9) .. (45.5,124.02) -- cycle ;
\draw [line width=2.25] (134.02, 90.87) circle [radius=30];
\draw (134.02, 165) circle [radius=25];
\draw [line width=2.25]    (134.02,35.5) -- (134.02,60.87) ;

\end{tikzpicture}

}\raise 7ex \hbox{,}
\end{figure}

\noindent
or else it is isotopic to $\omega^\prime$. The first web is zero by \eqref{eq:skein3}.

Now suppose $W$ is a connected web containing the marked point and having $n>1$ internal trivalent vertices. Assume the inductive hypothesis that any connected web containing the marked point and having $k<n$ internal vertices is in the span of $\{\omega'z^j: j \in \N_0\}$. Now, take a spanning tree of $W$, not considering the marked point as a vertex. There exists a neighborhood of the tree that is homeomorphic to a disk, with its complement being homeomorphic to a disk with two punctures. This splits the web into two diagrams. The first diagram $B$ is a trivalent tree containing the $n$ internal trivalent vertices and meeting the boundary of the disk at $n+2$ univalent endpoints. The second diagram $D$ is a planar matching on the twice punctured disk with $n+2$ marked points on the outer boundary and one marked point at a puncture.

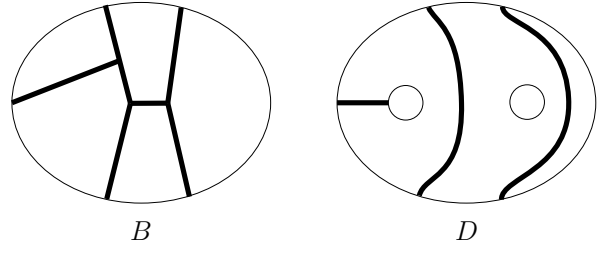
\begin{figure}[h!]
    \centering
    \resizebox{8cm}{!}{

\begin{tikzpicture}[x=0.75pt,y=0.75pt,yscale=-1,xscale=1]

\draw   (272.2,95.63) .. controls (272.2,63.25) and (306.05,37) .. (347.8,37) .. controls (389.55,37) and (423.4,63.25) .. (423.4,95.63) .. controls (423.4,128.01) and (389.55,154.26) .. (347.8,154.26) .. controls (306.05,154.26) and (272.2,128.01) .. (272.2,95.63) -- cycle ;
\draw   (302.2,95.63) .. controls (302.2,90.05) and (306.72,85.53) .. (312.3,85.53) .. controls (317.88,85.53) and (322.4,90.05) .. (322.4,95.63) .. controls (322.4,101.21) and (317.88,105.73) .. (312.3,105.73) .. controls (306.72,105.73) and (302.2,101.21) .. (302.2,95.63) -- cycle ;
\draw   (373.2,95.23) .. controls (373.2,89.65) and (377.72,85.13) .. (383.3,85.13) .. controls (388.88,85.13) and (393.4,89.65) .. (393.4,95.23) .. controls (393.4,100.81) and (388.88,105.33) .. (383.3,105.33) .. controls (377.72,105.33) and (373.2,100.81) .. (373.2,95.23) -- cycle ;
\draw [line width=2.25]    (368.93,39.51) .. controls (368.53,49.77) and (405.53,51.18) .. (407.53,93.18) .. controls (409.53,135.18) and (368.04,138.86) .. (368.27,152.18) ;
\draw   (82.2,95.63) .. controls (82.2,63.25) and (116.05,37) .. (157.8,37) .. controls (199.55,37) and (233.4,63.25) .. (233.4,95.63) .. controls (233.4,128.01) and (199.55,154.26) .. (157.8,154.26) .. controls (116.05,154.26) and (82.2,128.01) .. (82.2,95.63) -- cycle ;
\draw [line width=2.25]    (82.2,95.63) -- (144.6,71.18) ;
\draw [line width=2.25]    (136.8,38.67) -- (151.47,95.96) ;
\draw [color={rgb, 255:red, 0; green, 0; blue, 0 }  ,draw opacity=1 ][line width=2.25]    (151.47,95.96) -- (137.6,151.85) ;
\draw [line width=2.25]    (272.2,95.63) -- (302.2,95.63) ;
\draw [line width=2.25]    (181.13,40) -- (173.27,95.85) ;
\draw [line width=2.25]    (325.6,39.85) .. controls (329.6,48.85) and (343.93,51.85) .. (344.93,95.85) .. controls (345.93,139.85) and (323.27,137.51) .. (320.27,149.85) ;
\draw [line width=2.25]    (173.27,95.85) -- (185.93,150.51) ;
\draw [line width=2.25]    (151.47,95.96) -- (173.27,95.85) ;

\draw (150,163.9) node [anchor=north west][inner sep=0.75pt]    {$B$};
\draw (340,163.69) node [anchor=north west][inner sep=0.75pt]    {$D$};

\end{tikzpicture}

}
    \caption{An example of a web decomposition where $n=3.$}
    \label{fig:placeholder}
\end{figure}

The planar matching $D$ must have an edge $E$ traveling from the puncture to the boundary of the 2-punctured disk. Consider $D\setminus E$ as a diagram on the 2-punctured disk. We have that $D\setminus E$ is a planar matching with $n+1>2$ endpoints, implying that $D \setminus E$ is comprised of at least $s=\frac{n+1}{2}\geq2$ strands on the 2-punctured disk. The $s$ strands separate the 2-punctured disk into $s+1\geq3$ regions. Since there are only 2 punctures on the disk, there must be at least one unpunctured region on the disk. Recall that all trivalent trees with $n$ vertices are equivalent up to I-H moves. There exists a tree $B'$ such that replacing $B$ with $B'$ will enclose the unpunctured face of $D$ with a monogon or bigon. Thus, we can use I-H skein relations from \eqref{eq:skein3} to rewrite $W$ as a reducible web plus webs with fewer vertices. This process reduces our web to a linear combination of webs with fewer vertices which, by induction, is in the span of $\{\omega'z^j : j \in \N_0\}$. Thus, we have established that $\{\omega'z^j : j \in \N_0\}$ is a spanning set for $A_{1,0}^Y.$

\begin{figure}[h!]
    \centering
    \resizebox{8cm}{!}{

\begin{tikzpicture}[x=0.75pt,y=0.75pt,yscale=-1,xscale=1]

\draw   (71.53,120.3) .. controls (71.53,87.92) and (105.38,61.67) .. (147.13,61.67) .. controls (188.89,61.67) and (222.73,87.92) .. (222.73,120.3) .. controls (222.73,152.68) and (188.89,178.93) .. (147.13,178.93) .. controls (105.38,178.93) and (71.53,152.68) .. (71.53,120.3) -- cycle ;
\draw   (101.53,120.3) .. controls (101.53,114.72) and (106.06,110.2) .. (111.63,110.2) .. controls (117.21,110.2) and (121.73,114.72) .. (121.73,120.3) .. controls (121.73,125.88) and (117.21,130.4) .. (111.63,130.4) .. controls (106.06,130.4) and (101.53,125.88) .. (101.53,120.3) -- cycle ;
\draw   (172.53,119.9) .. controls (172.53,114.32) and (177.06,109.8) .. (182.63,109.8) .. controls (188.21,109.8) and (192.73,114.32) .. (192.73,119.9) .. controls (192.73,125.48) and (188.21,130) .. (182.63,130) .. controls (177.06,130) and (172.53,125.48) .. (172.53,119.9) -- cycle ;
\draw [line width=2.25]    (71.53,120.3) -- (101.53,120.3) ;
\draw [line width=2.25]    (136,61.63) .. controls (136.2,71.73) and (135.2,94.73) .. (135.2,119.73) .. controls (135.2,144.73) and (135.7,165.23) .. (135.5,178.63) ;
\draw [line width=2.25]    (158,62.63) .. controls (157.9,73.23) and (158.4,93.73) .. (158.4,118.73) .. controls (158.4,143.73) and (157.2,165.73) .. (157.48,177.9) ;
\draw [color={rgb, 255:red, 208; green, 2; blue, 27 }  ,draw opacity=1 ][line width=2.25]    (147.25,50.41) -- (147.25,33.34) ;
\draw   (262.03,120.63) .. controls (262.03,88.25) and (295.88,62) .. (337.63,62) .. controls (379.39,62) and (413.23,88.25) .. (413.23,120.63) .. controls (413.23,153.01) and (379.39,179.26) .. (337.63,179.26) .. controls (295.88,179.26) and (262.03,153.01) .. (262.03,120.63) -- cycle ;
\draw   (292.03,120.63) .. controls (292.03,115.06) and (296.56,110.53) .. (302.13,110.53) .. controls (307.71,110.53) and (312.23,115.06) .. (312.23,120.63) .. controls (312.23,126.21) and (307.71,130.73) .. (302.13,130.73) .. controls (296.56,130.73) and (292.03,126.21) .. (292.03,120.63) -- cycle ;
\draw   (363.03,120.23) .. controls (363.03,114.66) and (367.56,110.13) .. (373.13,110.13) .. controls (378.71,110.13) and (383.23,114.66) .. (383.23,120.23) .. controls (383.23,125.81) and (378.71,130.33) .. (373.13,130.33) .. controls (367.56,130.33) and (363.03,125.81) .. (363.03,120.23) -- cycle ;
\draw [line width=2.25]    (262.03,120.63) -- (292.03,120.63) ;
\draw [line width=2.25]    (400.4,87.63) .. controls (391.4,100.63) and (387.4,131.63) .. (398.9,155.63) ;
\draw  [draw opacity=0][line width=2.25]  (136.05,62.59) .. controls (136.02,62.28) and (136,61.96) .. (136,61.63) .. controls (136,56.05) and (140.95,51.53) .. (147.05,51.53) .. controls (153.16,51.53) and (158.1,56.05) .. (158.1,61.63) .. controls (158.1,61.96) and (158.09,62.28) .. (158.05,62.59) -- (147.05,61.63) -- cycle ; \draw  [color={rgb, 255:red, 208; green, 2; blue, 27 }  ,draw opacity=1 ][line width=2.25]  (136.05,62.59) .. controls (136.02,62.28) and (136,61.96) .. (136,61.63) .. controls (136,56.05) and (140.95,51.53) .. (147.05,51.53) .. controls (153.16,51.53) and (158.1,56.05) .. (158.1,61.63) .. controls (158.1,61.96) and (158.09,62.28) .. (158.05,62.59) ;  
\draw [color={rgb, 255:red, 208; green, 2; blue, 27 }  ,draw opacity=1 ][line width=2.25]    (146.68,190.09) -- (146.68,207.16) ;
\draw  [draw opacity=0][line width=2.25]  (135.48,177.9) .. controls (135.45,178.22) and (135.43,178.54) .. (135.43,178.87) .. controls (135.43,184.44) and (140.38,188.97) .. (146.48,188.97) .. controls (152.58,188.97) and (157.53,184.44) .. (157.53,178.87) .. controls (157.53,178.54) and (157.52,178.22) .. (157.48,177.9) -- (146.48,178.87) -- cycle ; \draw  [color={rgb, 255:red, 208; green, 2; blue, 27 }  ,draw opacity=1 ][line width=2.25]  (135.48,177.9) .. controls (135.45,178.22) and (135.43,178.54) .. (135.43,178.87) .. controls (135.43,184.44) and (140.38,188.97) .. (146.48,188.97) .. controls (152.58,188.97) and (157.53,184.44) .. (157.53,178.87) .. controls (157.53,178.54) and (157.52,178.22) .. (157.48,177.9) ;  
\draw [color={rgb, 255:red, 208; green, 2; blue, 27 }  ,draw opacity=1 ][line width=2.25]    (398.9,155.63) .. controls (410.32,180.3) and (425,150.79) .. (425.67,122.31) .. controls (426.33,93.84) and (414.15,71.64) .. (400.4,87.63) ;
\draw [color={rgb, 255:red, 208; green, 2; blue, 27 }  ,draw opacity=1 ][line width=2.25]    (425.67,122.31) -- (438.67,122.31) ;

\end{tikzpicture}

}
    \caption{An example of a bigon and monogon resolution}
    \label{fig:bigon_monogon}
\end{figure}
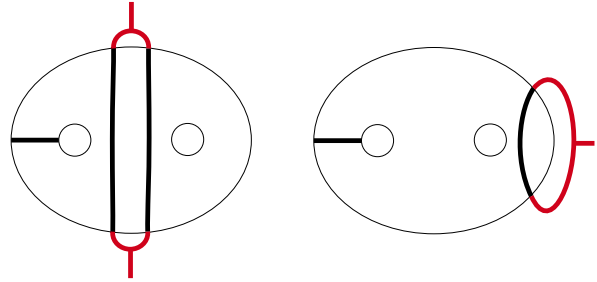

To see that $\{\omega'z^j : j \in \N_0\}$ is linearly independent, we can check that its image under the map $\varphi: A_{1,0}^Y \to A_{2, 0}$ from \eqref{eq:phimap} is linearly independent. We have:
\begin{align*}\varphi(\omega'z^j) &= \varphi(\omega')\cdot \varphi(z^j) = \varphi(\omega')(\varphi(z))^j=(zv + \frac{1}{[2]}z^2 u')(z^2 -1)^j,
\end{align*}

where $v,u'$ are the elements of $A_{2,0}$ pictured in Figure \ref{fig:A_2_0}.

Written in the basis for $A_{2,0}$ given by the elements listed in \eqref{eq:A_2_0}, the term with the most core loops in $\varphi(\omega'z^j)$ has $2j+2$ core loops. Thus, $\varphi$ sends $\{\omega' z^j: j \in \mathbb{N}_0\}$ to a linearly independent subset of $A_{2,0}$ and so we conclude that $\{\omega' z^j: j \in \mathbb{N}_0\}$ is linearly independent.
\end{proof}

\begin{proposition}\label{prop:A11 basis}
The set $\{u^k, s, \omega z^j : k \in \Z, j \in \N_0\}$ is a basis for $A_{1, 1}^Y$ where $u, s, \omega$ and $z$ are as depicted below.
\end{proposition}

\begin{center}
\includegraphics[width=0.8\linewidth]{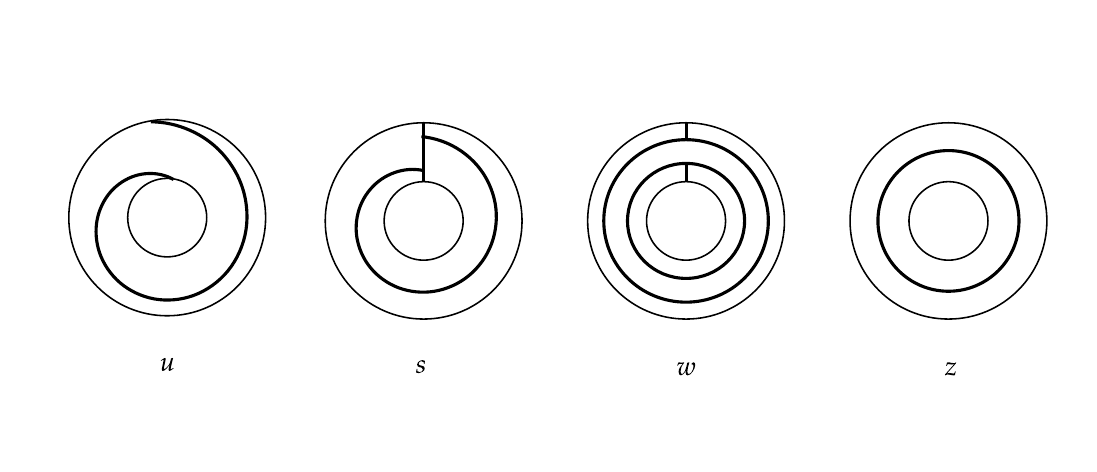}
\end{center}

\begin{proof}
We will first prove that $\{ u^k, su^l, \omega z^j \}$ is a spanning set before shrinking this set to a basis. Let $W \in A_{1,1}^Y$ be a web. First observe that any crossings can be eliminated using \eqref{eq:skein3}. The resulting linear combination of webs are all crossingless disjoint unions of connected webs. Every such crossingless web must have at most two components connected to the marked points. Given a crossingless web, we have two cases:

\bigskip
\textbf{Case 1}: Disconnected marked points
\bigskip

We will assume the two marked points do not lie on the same connected component. Consider the $\C$-linear map $p: A_{1,0}^Y \otimes A_{0,1}^Y \to A_{1,1}^Y$ by concatenation of diagrams given in Figure \ref{fig:cut_point}.

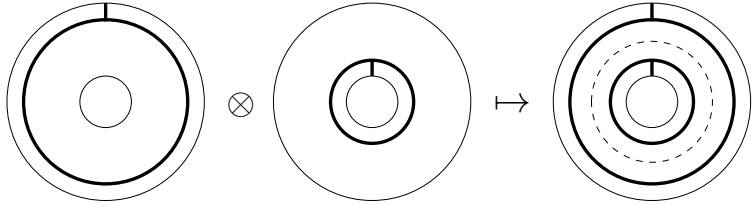
\begin{figure}[h!]
    \centering
    \resizebox{10cm}{!}{

\begin{tikzpicture}[x=0.75pt,y=0.75pt,yscale=-1,xscale=1]

\draw   (509.5,120.02) .. controls (509.5,71.13) and (549.13,31.5) .. (598.02,31.5) .. controls (646.9,31.5) and (686.53,71.13) .. (686.53,120.02) .. controls (686.53,168.9) and (646.9,208.53) .. (598.02,208.53) .. controls (549.13,208.53) and (509.5,168.9) .. (509.5,120.02) -- cycle ;
\draw   (574.91,120.02) .. controls (574.91,107.25) and (585.25,96.91) .. (598.02,96.91) .. controls (610.78,96.91) and (621.12,107.25) .. (621.12,120.02) .. controls (621.12,132.78) and (610.78,143.12) .. (598.02,143.12) .. controls (585.25,143.12) and (574.91,132.78) .. (574.91,120.02) -- cycle ;
\draw  [line width=2.25]  (524.41,120.02) .. controls (524.41,79.37) and (557.37,46.41) .. (598.02,46.41) .. controls (638.67,46.41) and (671.62,79.37) .. (671.62,120.02) .. controls (671.62,160.67) and (638.67,193.62) .. (598.02,193.62) .. controls (557.37,193.62) and (524.41,160.67) .. (524.41,120.02) -- cycle ;
\draw [line width=2.25]    (598.02,31.5) -- (598.02,47.87) ;
\draw [line width=2.25]    (598.02,82.75) -- (598.02,96.91) ;
\draw  [line width=2.25]  (560.75,120.02) .. controls (560.75,99.44) and (577.44,82.75) .. (598.02,82.75) .. controls (618.6,82.75) and (635.28,99.44) .. (635.28,120.02) .. controls (635.28,140.6) and (618.6,157.28) .. (598.02,157.28) .. controls (577.44,157.28) and (560.75,140.6) .. (560.75,120.02) -- cycle ;
\draw  [dash pattern={on 4.5pt off 4.5pt}] (543.78,120.02) .. controls (543.78,90.06) and (568.06,65.78) .. (598.02,65.78) .. controls (627.97,65.78) and (652.25,90.06) .. (652.25,120.02) .. controls (652.25,149.97) and (627.97,174.25) .. (598.02,174.25) .. controls (568.06,174.25) and (543.78,149.97) .. (543.78,120.02) -- cycle ;
\draw   (19.5,120.02) .. controls (19.5,71.13) and (59.13,31.5) .. (108.02,31.5) .. controls (156.9,31.5) and (196.53,71.13) .. (196.53,120.02) .. controls (196.53,168.9) and (156.9,208.53) .. (108.02,208.53) .. controls (59.13,208.53) and (19.5,168.9) .. (19.5,120.02) -- cycle ;
\draw   (84.91,120.02) .. controls (84.91,107.25) and (95.25,96.91) .. (108.02,96.91) .. controls (120.78,96.91) and (131.12,107.25) .. (131.12,120.02) .. controls (131.12,132.78) and (120.78,143.12) .. (108.02,143.12) .. controls (95.25,143.12) and (84.91,132.78) .. (84.91,120.02) -- cycle ;
\draw  [line width=2.25]  (34.41,120.02) .. controls (34.41,79.37) and (67.37,46.41) .. (108.02,46.41) .. controls (148.67,46.41) and (181.62,79.37) .. (181.62,120.02) .. controls (181.62,160.67) and (148.67,193.62) .. (108.02,193.62) .. controls (67.37,193.62) and (34.41,160.67) .. (34.41,120.02) -- cycle ;
\draw [line width=2.25]    (108.02,31.5) -- (108.02,47.87) ;
\draw   (258.5,120.02) .. controls (258.5,71.13) and (298.13,31.5) .. (347.02,31.5) .. controls (395.9,31.5) and (435.53,71.13) .. (435.53,120.02) .. controls (435.53,168.9) and (395.9,208.53) .. (347.02,208.53) .. controls (298.13,208.53) and (258.5,168.9) .. (258.5,120.02) -- cycle ;
\draw   (323.91,120.02) .. controls (323.91,107.25) and (334.25,96.91) .. (347.02,96.91) .. controls (359.78,96.91) and (370.12,107.25) .. (370.12,120.02) .. controls (370.12,132.78) and (359.78,143.12) .. (347.02,143.12) .. controls (334.25,143.12) and (323.91,132.78) .. (323.91,120.02) -- cycle ;
\draw [line width=2.25]    (347.02,82.75) -- (347.02,96.91) ;
\draw  [line width=2.25]  (309.75,120.02) .. controls (309.75,99.44) and (326.44,82.75) .. (347.02,82.75) .. controls (367.6,82.75) and (384.28,99.44) .. (384.28,120.02) .. controls (384.28,140.6) and (367.6,157.28) .. (347.02,157.28) .. controls (326.44,157.28) and (309.75,140.6) .. (309.75,120.02) -- cycle ;

\draw (455,115.4) node [anchor=north west][inner sep=0.75pt]  [font=\huge]  {$\mapsto $};
\draw (215,110.4) node [anchor=north west][inner sep=0.75pt]  [font=\huge]  {$\otimes $};

\end{tikzpicture}

}
    \caption{The map from $A_{1,0}^Y \otimes A_{0,1}^Y$ to $A_{1,1}^Y$}
    \label{fig:cut_point}
\end{figure}

The image of $p$ lies in the subspace of $A_{1,1}^Y$ spanned by all crossingless webs where the two marked points are in two separate connected components. Given a web $W$ in this subspace, there exists a closed path around the annular hole that doesn't cross $W$. In other words, we can split $W$ into two diagrams $W_1 \in A_{1,0}^Y$ and $W_2 \in A_{0,1}^Y$. Then, $p(W_1 \otimes W_2) = W$. Thus, $p$ surjects onto this subspace.

By \thref{A_1_0_lemma}, $\{\omega' z^i \otimes \omega' z^j : i,j \in \N_0\}$ is a basis for $A_{1,0}^Y \otimes A_{0,1}^Y$. Since $p(\omega'z^i \otimes \omega' z^j) = \omega z^{i+j}$, the image of $p$ is precisely $\operatorname{span}(\{\omega z^j : j \in \N_0\})$, which is contained in the span of our claimed spanning set.

\bigskip
\textbf{Case 2}: Connected marked points
\bigskip

If the two marked points lie on the same connected component of a crossingless web, then every other component must be a closed web contained in some disk. Thus, the web reduces to a scalar multiple of a web with only the single component connecting the two marked points.

We now prove that any crossingless web $W$ which has a component with univalent vertices at both punctures is in the span of $\{u^k,su^l,\omega z^j\}.$ We will induct on the number $n$ of internal vertices the component has, noting that $n$ must be even. In particular, assume that every such web component with fewer than $n$ internal vertices is in the span of $\{ u^k, su^l, \omega z^j \}$. The base case is when $n = 0.$ In this case, the web contains a component that is an embedded curve traveling between the two punctures. In this case, $W$ is a scalar multiple of $u^k$ for some $k \in \Z$.

For $n \geq 2$, we follow a similar argument as in \thref{A_1_0_lemma}.
Use a spanning tree of the graph to split $W$ into two diagrams, $B$ and $D$. The first diagram $B$ is a trivalent tree with $n$ internal trivalent vertices contained in a disk and having $n + 2$ end points on the boundary of the disk. The second diagram $D$ is a planar matching on the twice punctured disk with $n+2$ endpoints on the boundary of the disk and one endpoint at each puncture.

The diagram $D$ contains an edge $E_1$ traveling from the puncture to the boundary of the disk and an edge $E_2$ traveling from the other puncture to the boundary of the disk. Consider the diagram $D\setminus{(E_1 \cup E_2)}$ on the 2-punctured disk. This diagram is a planar matching on the 2-punctured disk having $n$ endpoints on the boundary of the disk.

If $n=2,$ the planar matching contains only 1 strand which separates the 3 punctured sphere into 2 regions. If one of these regions happens to be unpunctured, we can use the technique of the proof of \thref{A_1_0_lemma} to apply I-H skein relations to the diagram $B$ to rewrite $W$ as a web with a face that may be reduced plus terms with fewer vertices. If both regions are punctured, then we may deduce that $W$ is diffeomorphic to $s$ and so must be able to be written as a Dehn twist of $s$, which means $W=su^l$ for some $l \in \mathbb{Z}.$ 

Now if $n>2$ then this planar matching contains at least 2 strands, which separate the 2-punctured disk into at least 3 regions and at least one of these must be unpunctured. As in the proof of \thref{A_1_0_lemma}, we can apply I-H skein relations to the diagram $B$ to rewrite the web $W$ as a web with a face that may be reduced plus terms with fewer vertices. Thus, we deduce by induction that $\{ u^k, su^l, \omega z^j \}$ is a spanning set for $A_{1,1}^Y.$

\bigskip

Next, we shrink our spanning set. By direct computation we see that

\begin{equation}\label{eq:su}
su = s + \frac{1}{[2]}(u^2 - 1).
\end{equation}

Multiplying \eqref{eq:su} by $u^{-1}$ and rearranging yields $su^{-1}=s-\frac{1}{[2]}(u-u^{-1}).$ Using these two equations, we can reduce our spanning set to our claimed basis $\{ u^k, s, \omega z^j \}$. Linear independence is seen by looking at the image of these elements under the map $\varphi: A_{1,1}^Y \to A_{2,2}$.

\begin{align*}
    \varphi(u^k) &= u^k + \frac{1}{[2]}f_4 \\
    \varphi(s) &= f_1 + \frac{1}{[2]}(f_2z + f_3z) + \frac{1}{[2]^2}f_4 (z^2+2) + \frac{1}{[2]}(u + 1) \\
    \varphi(\omega z^n) &= (f_1z^2 + \frac{1}{[2]}f_2z^3 + \frac{1}{[2]}f_3z^3 + \frac{1}{[2]^2} f_4 z^4)(z^2 -1)^n
\end{align*}

Since $\{u^{k/2}, f_1z^n, f_2z^n, f_3z^n, f_4z^n\}$ form a basis for $A_{2,2}$, these are all linearly independent.

\end{proof}

The last argument in the proof provides the following corollary.

\begin{corollary}
The map $\varphi: A_{1,1}^Y \to A_{2,2}$ is injective.
\end{corollary}

We will use the following fact in the next section.

\begin{corollary}
The algebra $A_{1,1}^Y$ is commutative.
\end{corollary}

\begin{proof}
It suffices to check that the basis elements from Proposition \ref{prop:A11 basis} commute with each other. Since $u$ is seen to be central by simple isotopies, and since copies of the diagram $z$ may be isotoped over other elements, it will suffice to check that $s$ and $\omega$ commute with each other.

To show $s$ and $\omega$ commute, we diagrammatically reduce $s\omega$ and $\omega s$ to the following linear combination of basis elements for $A_{1,1}^Y:$

\begin{equation*}
s\omega=\frac{1}{[2]}z \omega+\frac{2-[3]}{[2]} \omega=\omega s.
\end{equation*}
\end{proof}


\section{Threading and Transparency}\label{sec: threading and transparency}

In this section, we recall the notion of transparency in a skein module and define a family of polynomials $P_n$ with the goal of producing transparent elements in $\Sq^Y(M)$ when $q$ is a root of unity. In \cite{ESS_2025} a family of polynomials $Y_n$ was produced in order to find a presentation for the $\rm{PGL}_2$ skein algebra of the torus. It turns out that $P_n$ and $Y_n$ are closely related and in fact differ only by a constant: $Y_n=P_n-1.$ We will use this relationship when studying the Frobenius homomorphism for the $\rm{PGL}_2$ skein algebra of the torus in Section \ref{sec: Frobenius homomorphism for torus}.

The strategy of constructing transparent elements here is largely similar to the ones used in \cite{L_2015,MS17,BH24,BBHHMP25,HW26}, which focus on computations in the annulus. The feature specific to the $\rm{PGL}_2$ skein theory is the construction of a twist-invariant element $\chi-1 \in A_{1,1}^Y$ which plays a main role in our proof. 

\begin{definition}\label{defn:Pn-family}
We define the following sequence of polynomials $P_n\in \Z[x]$ as follows:
\begin{itemize}
    \item $P_0(x) = 3$
    \item $P_1(x) = x$
    \item $P_2(x) = x^2 - 2x$
    \item $P_n(x) = x(P_{n-1}(x) - P_{n-2}(x)) + P_{n-3}(x)$ for $n \geq 3.$
\end{itemize}
\end{definition}

These polynomials satisfy the property that for any $\lambda \in \C^{\times}$, $P_n(\lambda + 1 + \lambda^{-1}) = \lambda^n + 1 + \lambda^{-n}$.
They also satisfy the relation $P_n(x^2 - 1) = T_n(x)^2 - 1$. These two properties can be proven by induction and may also be taken to be defining properties. From the first property, one may also deduce the identity

\begin{equation}\label{eq:P composition}
P_{n}(P_k(x))=P_{nk}(x).
\end{equation}

\bigskip

\begin{definition}
Suppose $K$ is a framed knot in $M$. Let $K^{[x^n]}$ be the link obtained by
drawing $n$ parallel copies of $K$. We say that $K^{[x^n]}$ is the knot $K$
threaded by the monomial $x^n$. Now if $p(x) = \sum_{i=0} ^m a_ix^i$ is a
polynomial, we say that $K^{[p(x)]} = \sum_{i=0} ^m a_iK^{[x^i]}$ is
$K$ threaded by $p(x)$. Moreover, if $L = L_1 \sqcup \cdots 
\sqcup L_k$ is a link with components $L_1, \dots L_k$, then we define

\begin{equation}\label{eq:threading}
L^{[p(x)]} := \sum_{i_1, \dots, i_k} a_{i_1}\cdots a_{i_k}\coprod_{j=1} ^k L_j ^{[x^{i_j}]}.    
\end{equation}
\end{definition}

We define the following two algebra homomorphisms from $A_{0,0}^Y$ to $A_{1,1}^Y$. Given a diagram $D \in A_{0,0}^Y$, we define $D_*$ to be the diagram in $A_{1,1}^Y$ created by placing $D$ over the identity element $e \in A_{1,1}^Y.$ We define $D^*$ to be the diagram in $A_{1,1}^Y$ created by placing $D$ under $e$. The result of applying each map to the loop $z \in A_{0,0}^Y$ is shown below.

\[ z^* = {\zupper}, \quad z_* = {\zlower} \]

\begin{definition}
We say that a polynomial $P \in \C[x]$ is transparent if it satisfies the property $P(z)_*=P(z)^*$ or, equivalently, $P(z_*)=P(z^*),$ since each homomorphism is an algebra homomorphism.
\end{definition}

If $P$ is a transparent polynomial, then the element of $\Y(M)$ obtained by replacing a framed knot $K$ in $\Y(M)$ by its threading $K^{[P(x)]}$ can pass through any other element in the skein module. If $M=\Sigma \times (0,1)$ is a thickened surface then an element threaded by a transparent polynomial $P$ becomes a central element of the skein algebra $\Y(\Sigma).$

We will soon obtain an identity specific to our polynomial $P_n$ which will be used to show that $P_n$ is a transparent polynomial when $q$ is an appropriate root of unity.

To aid in our computations we will examine the expressions of $z^*$ and $z_*$ using the basis of $A_{1,1}^Y.$ Using the crossing removal skein relation from \eqref{eq:skein3} we can reduce $z^*$ and $z_*$ into basis elements of $A_{1,1}^Y$:
\begin{align*}
z^*&=q^2u+q^{-2}u^{-1}+([2]s-u)=q^2u+q^{-2}u^{-1}+\chi \\
z_*&=q^{-2}u+q^2u^{-1}+([2]s-u)=q^{-2}u+q^2u^{-1}+\chi,
\end{align*}
where we have defined $\chi=[2]s-u.$

The following proposition shows that the element $\chi-1$ is invariant under multiplication by $u$ (which one can interpret as being invariant under Dehn-twists).

\begin{proposition}
The following equality holds in $A_{1,1}^Y:$
\begin{equation}\label{eq:twist invariance}
u(\chi-1)=\chi-1.
\end{equation}
\end{proposition}

\begin{proof}

We observe
\begin{align*}
u\chi&=u([2]s-u)\\
&=[2]su-u^2\\
&=[2](s+\frac{1}{[2]}(u^2-1))-u^2\\
&=[2]s-1\\
&=\chi-1+u,
\end{align*}
where the third equality comes from applying \eqref{eq:su}. From this, we deduce that
\begin{equation}
u(\chi-1)=\chi-1.
\end{equation}

\end{proof}

\begin{lemma}\label{lem:Pn Rn}
For any $n \geq 0,$ we have that
\begin{align*}
P_n(z^*)&=q^{2n}u^n+q^{-2n}u^{-n}+R_n(\chi-1)\\
P_n(z_*)&=q^{-2n}u^n+q^{2n}u^{-n}+R_n(\chi-1),
\end{align*}
where $R_n$ is a degree $n$ polynomial with constant coefficient equal to 1.
\end{lemma}
\begin{proof}
We will prove the statement by induction on $n,$ making use of the recursive formula for $P_n.$ We first check manually that the property holds for the three base cases: $n=0,1,2.$ We have that $P_0(z^*)=3=q^0u^0+q^{-0}u^{-0}+1,$ and $P_0(z_*)=3=q^{-0}u^0+q^0u^{-0}+1$  so the property holds for $n=0$ with $R_0=1.$

As computed above, $P_1(z^*)=z^*=q^{2}u+q^{-2}u^{-1}+([2]s-u)=q^{2}u+q^{-2}u^{-1}+\chi.$ We have $P_1(z_*)$ is the same expression but with $q$ replaced by $q^{-1}.$ Thus, the property holds for $n=1$ with $R_1(x)=x+1.$

Next, we have
\begin{align*}
P_2(z^*)=(z^*)^2-2(z^*)&=(q^2u+q^{-2}u^{-1}+\chi)^2-2(q^2u+q^{-2}u^{-1}+\chi)\\
&=q^4u^2+q^{-4}u^{-2}+(\chi-1)^2+2(q^2+q^{-2})(\chi-1)+1.
\end{align*}

We have that $P_2(z_*)$ is the same expression with $q$ replaced with $q^{-1}$ and see that the property holds for $R_2(x-1)=(x-1)^2+2(q^{2}+q^{-2})(x-1)+1.$

Now for $n \geq 3$ we recursively define $R_n(x)$ by the following.
\[R_n(x) = (q^2 + q^{-2} + x + 1) (R_{n-1}(x) - R_{n-2}(x)) + (q^{2n - 2} + q^{-2n + 2} - q^{2n - 4} - q^{-2n + 4})x + R_{n-3}(x) \]

We note that $R_n(x)$ is invariant under replacing $q$ by $q^{-1}.$ We also note it is evident by induction that the constant term of $R_n(x)$ is $1$ for all $n \geq 0.$ Using this last property as well as \eqref{eq:twist invariance}, we obtain 
\begin{equation}\label{eq:uRn}
uR_n(\chi-1)=R_n(\chi-1)+u-1.
\end{equation}

We now verify that $P_n(z^*)$ satisfies the desired property by using the recursive definition of $P_n(x)$ and the inductive assumption.

We have that
\begin{align*}
    P_n(z^*) &= z^* (P_{n-1}(z^*) - P_{n-2}(z^*)) + P_{n-3}(z^*) \\
    &= (q^2 u + q^{-2} u^{-1} + \chi) (q^{2n - 2} u^{n-1} + q^{-2n + 2} u^{-n + 1} + R_{n-1}(\chi - 1) \\ & \quad - q^{2n - 4} u^{n-2} - q^{-2n + 4} u^{-n + 2} - R_{n-2}(\chi - 1)) \\ & \quad + q^{2n - 6} u^{n-3} + q^{-2n + 6} u^{-n + 3} + R_{n-3}(\chi - 1) \\ 
    &= q^{2n} u^n + q^{-2n} u^{-n} + (q^2 + q^{-2} + \chi) (R_{n-1}(\chi - 1) - R_{n-2}(\chi - 1)) \\ & \quad + (q^{2n - 2} + q^{-2n + 2} - q^{2n - 4} - q^{-2n + 4})(\chi - 1) + R_{n-3}(\chi - 1)\\
    &=q^{2n}u^n+q^{-2n}u^{-n}+R_{n}(\chi-1),
\end{align*}
as required. Here the first equality is given by the recursive definition of $P_n,$ the second equality is given by the inductive assumption, the third equality is deduced by the property in equation \eqref{eq:uRn}, and the final equality is given by the recursive definition of $R_n.$

The expression for $P_n(z_*)$ is the same but with $q$ replaced by $q^{-1}.$ Since $R_n(\chi-1)$
is invariant under this replacement, we have obtained the required identities.

\end{proof}

\begin{proposition}\thlabel{P_n_property}
$P_n(z^*) - P_n(z_*) = (q^{2n} - q^{-2n})(u^n - u^{-n})$
\end{proposition}
\begin{proof}
We apply Lemma \ref{lem:Pn Rn} to subtract $P_n(z_*)$ from $P_n(z^*).$ The terms involving $R_n(\chi-1)$ cancel and we are left with the desired identity.
\end{proof}

\begin{theorem}\label{Thm:transparent}
A non-constant polynomial $p(x) \in \C[x]$ is transparent when the quantum parameter is set to $q=\xi$ if and only if

\begin{enumerate}
   \item $\xi$ is a root of unity, and
   \item $p \in \C[P_N(x)]$, where $N$ is the order of $\xi^4$.
\end{enumerate}
\end{theorem}

\begin{proof}
Suppose (1) and (2) hold. Then we apply \thref{P_n_property}. Since $\xi^{4N} = 1$, we have that
\[P_N(z^*) - P_N(z_*) = (\xi^{2N} - \xi^{-2N})(u^N - u^{-N}) = 0.\]
Thus, $P_N$ is transparent. Since the homomorphisms $(-)^*$ and $(-)_*$ are algebra homomorphisms, we then have that every element of $\C[P_N(x)]$ is transparent as well.

Now, suppose that $p$ is a transparent polynomial which is not constant. Then $p \in \C[x]$ satisfies $p(z^*) - p(z_*) = 0$. Note that the polynomials $P_i(x)$ form a basis for $\C[x]$ since $P_i$ is a degree $i$ polynomial for each $i \geq 0$. Thus, we can uniquely represent $p$ as $p(x) = \sum_{i=0}^{m} c_i P_i(x)$. By \thref{P_n_property} we have:
\begin{align*}
    p(z^*) - p(z_*) = \sum_{i=0}^{m} c_i (P_i(z^*) - P_i(z_*)) = \sum_{i=1}^{m} c_i (\xi^{2i} - \xi^{-2i})(u^i -u^{-i}) = 0
\end{align*}
Recall that the set $\{u^i: i \in \Z\}$ is linearly independent in $A_{1,1}^Y$ by Proposition \ref{prop:A11 basis}. Thus, $c_i (\xi^{2i} - \xi^{-2i}) = 0$ for each $i$. 

If $\xi$ is not a root of unity, then $\xi^{2i} \ne \xi^{-2i}$ for any $i\geq 1$. In this case, each $c_i = 0$ for $i \geq 1$ and $p(x) = c_0P_0(x)=3c_0$, contradicting the fact that $p$ is non-constant. Therefore, $\xi$ is a root of unity. Let $N$ denote the order of $\xi^4$. Then, $\xi^{2i} = \xi^{-2i}$ if and only if $N$ divides $i$. Thus, $c_i = 0$ for all $i$ not divisible by $N$ and $p$ is of the form
\[p(x) = \sum_{k=0}^{\frac{m}{N}} c_{kN} P_{kN}(x) = \sum_{k=0}^{\frac{m}{N}} c_{kN} P_{k}(P_N(x)) \in \C[P_N(x)] ,\]
where the last equality comes from the property \eqref{eq:P composition}.
\end{proof}

\section{Frobenius homomorphism for the torus}\label{sec: Frobenius homomorphism for torus}

In \cite{Bonahon_2015}, Bonahon and Wong constructed a Frobenius homomorphism

\begin{equation*}
Fr: \mathcal{S}_t(\Sigma) \rightarrow \mathcal{S}_q(\Sigma)
\end{equation*}
for any oriented surface $\Sigma$ when $q^2$ is a root of unity of order $n$ and $t=q^{n^2}.$ Their map is defined on framed links by the assignment $L \mapsto L^{[T_n]}.$  In this section we construct the Frobenius homomorphism for the closed torus in the $\rm{PGL}_2$ case, \begin{equation*}
Fr^{Y}:\mathcal{S}_t^Y(\T^2)\rightarrow \Y(\T^2) 
\end{equation*}when $q^2$ is a primitive $n$th root of unity and $t=q^{n^2}.$ Since in \cite{ESS_2025} it was shown that the map from the $\rm{PGL}_2$ skein algebra of the torus to the Kauffman bracket skein algebra of the torus is not injective, we will have to use some more ingredients from the $\rm{PGL}_2$ skein theory.  We will define $Fr^Y$ on a basis for $\mathcal{S}_t^Y(\T^2)$ and check that it is an algebra homomorphism by using the algebra presentation for $\Y(\T^2)$ from \cite{ESS_2025}.

We provide some useful information about the structure of $\mathcal{S}_q^Y(\T^2)$ from \cite{ESS_2025}.

\begin{definition}

\begin{itemize}
    \item[(a)]  Let $(k,l) \in \Z^2 \backslash \{(0,0)\}$ be a pair of coprime integers. Define $(k,l)\in \Y(\T^2)$ to be the element represented by the $(k,l)$ torus knot, the knot on the torus that traverses the meridian $k$ times and the longitude $l$ times with a slope of $k/l$, pictured in Figure \ref{fig:hexagon}. Since we are using unoriented skeins in $\Y(\T^2),$, the element represented by $(k,l)$ is the same as the element represented by $(-k,-l).$
    \item [(b)] Define $w \coloneq [2]h -(1,0) - (0,1) - (1,1) + 3\varnothing \in \Y(\T^2)$, where $h$ is the web defined in Figure \ref{fig:hexagon}.
    \item[(c)] For $(k,l) \in \Z^2 \backslash \{(0,0)\}$, not necessarily coprime, define $(k,l)_Y\in \Y(\T^2)$ to be $\left(\frac{k}{g},\frac{l}{g}\right)^{[Y_g]}$, where $g=\gcd(k,l)$ and $Y_g = P_g -1,$ where $P_g$ is the polynomial from Definition \ref{defn:Pn-family} and the superscript of $[Y_g]$ denotes the threading operation defined in \eqref{eq:threading}. We will use the convention that $\gcd(0,0)=0$ and $(0,0)_Y=2 \varnothing.$
\end{itemize}

\end{definition}

The following two propositions were proven in \cite{ESS_2025} and hold for any value of $q$ such that $q+q^{-1} \neq 0.$

\begin{figure}[h!]
\centering
\resizebox{10cm}{!}{

\tikzset{every picture/.style={line width=0.75pt}} 

\begin{tikzpicture}[x=0.75pt,y=0.75pt,yscale=-1,xscale=1]

\draw   (109,81) -- (259,81) -- (259,231) -- (109,231) -- cycle ;
\draw    (169,81) -- (109,231) ;
\draw    (229,81) -- (169,231) ;
\draw    (259,155) -- (229,231) ;
\draw    (139,81) -- (109,156) ;
\draw    (200,81) -- (139,231) ;
\draw    (259,81) -- (198.5,231) ;
\draw    (79,165) -- (107.08,156.57) ;
\draw [shift={(109,156)}, rotate = 163.3] [color={rgb, 255:red, 0; green, 0; blue, 0 }  ][line width=0.75]    (10.93,-3.29) .. controls (6.95,-1.4) and (3.31,-0.3) .. (0,0) .. controls (3.31,0.3) and (6.95,1.4) .. (10.93,3.29)   ;
\draw    (79,165) -- (108.17,229.18) ;
\draw [shift={(109,231)}, rotate = 245.56] [color={rgb, 255:red, 0; green, 0; blue, 0 }  ][line width=0.75]    (10.93,-3.29) .. controls (6.95,-1.4) and (3.31,-0.3) .. (0,0) .. controls (3.31,0.3) and (6.95,1.4) .. (10.93,3.29)   ;
\draw    (188,58) -- (257.1,80.38) ;
\draw [shift={(259,81)}, rotate = 197.95] [color={rgb, 255:red, 0; green, 0; blue, 0 }  ][line width=0.75]    (10.93,-3.29) .. controls (6.95,-1.4) and (3.31,-0.3) .. (0,0) .. controls (3.31,0.3) and (6.95,1.4) .. (10.93,3.29)   ;
\draw    (188,58) -- (227.26,80.02) ;
\draw [shift={(229,81)}, rotate = 209.29] [color={rgb, 255:red, 0; green, 0; blue, 0 }  ][line width=0.75]    (10.93,-3.29) .. controls (6.95,-1.4) and (3.31,-0.3) .. (0,0) .. controls (3.31,0.3) and (6.95,1.4) .. (10.93,3.29)   ;
\draw    (188,58) -- (199.07,79.23) ;
\draw [shift={(200,81)}, rotate = 242.45] [color={rgb, 255:red, 0; green, 0; blue, 0 }  ][line width=0.75]    (10.93,-3.29) .. controls (6.95,-1.4) and (3.31,-0.3) .. (0,0) .. controls (3.31,0.3) and (6.95,1.4) .. (10.93,3.29)   ;
\draw    (188,58) -- (170.27,79.46) ;
\draw [shift={(169,81)}, rotate = 309.56] [color={rgb, 255:red, 0; green, 0; blue, 0 }  ][line width=0.75]    (10.93,-3.29) .. controls (6.95,-1.4) and (3.31,-0.3) .. (0,0) .. controls (3.31,0.3) and (6.95,1.4) .. (10.93,3.29)   ;
\draw    (188,58) -- (140.81,80.15) ;
\draw [shift={(139,81)}, rotate = 334.86] [color={rgb, 255:red, 0; green, 0; blue, 0 }  ][line width=0.75]    (10.93,-3.29) .. controls (6.95,-1.4) and (3.31,-0.3) .. (0,0) .. controls (3.31,0.3) and (6.95,1.4) .. (10.93,3.29)   ;
\draw   (289.5,79) -- (440,79) -- (440,229.5) -- (289.5,229.5) -- cycle ;
\draw    (289.5,154.25) -- (440,154.25) ;
\draw    (364.75,79) -- (333.15,154.25) ;
\draw    (393.35,154.25) -- (363.25,229.5) ;

\draw (161,37) node [anchor=north west][inner sep=0.75pt]   [align=left] {{\fontfamily{ptm}\selectfont $k$ times}};
\draw (30,156) node [anchor=north west][inner sep=0.75pt]   [align=left] {{\fontfamily{ptm}\selectfont $l$ times}};

\end{tikzpicture}
}
\caption{A $(k,l)$ knot on the left and the hexagon $h$ on the right.}
\label{fig:hexagon}
\end{figure}

The first gives a description of the basis of $\Y(\T^2)$ which is similar to the Chebyshev basis of the torus given by Frohman and Gelca \cite{FG00}.

\begin{proposition}[\cite{ESS_2025}]\label{thm:torus basis} 
$\{(p,q)_Y, w, \varnothing: (p,q) \in \Z^2\backslash \{(0,0)\}\}$ forms a basis of $\mathcal{S}_q^Y(\T^2)$.
\end{proposition}

The next proposition provides for us the following analogue of the product-to-sum description of the skein algebra of the torus that Frohman and Gelca discovered for the Kauffman bracket case \cite{FG00}.

\begin{proposition}[\cite{ESS_2025}]\label{thm:multiplication table}
For any $k,l,r,s \in \mathbb{Z}$, the product of basis elements of $\Y(\T^2)$ is determined by the following relations:
	\[\begin{split}
	(k,l)_{Y} \cdot (r,s)_{Y} & = q^{2d}(k + r, l + s)_{Y} + q^{-2d}(k - r, l - s)_{Y} + \alpha w,\\
	(k,l)_{Y} \cdot w & = w \cdot (k,l)_{Y} = (-1)^{g_{1}}2w,\\
	w^{2} & = 2[2]^{2}w,\\
	\end{split}\]
    where the constants $\alpha,d,g_1$ are defined by the following formulas:

    \[\begin{split}
		d & = - \begin{vmatrix} k & l \\ r & s \end{vmatrix},\\
		\alpha & = \frac{(-1)^{g_{1} + g_{2}}2 - (-1)^{g_{3}}q^{2d} - (-1)^{g_{4}}q^{-2d}}{[2]^{2}},\\
		g_{1} & = \mathrm{gcd}(k,l),\\
		g_{2} & = \mathrm{gcd}(r,s),\\
		g_{3} & = \mathrm{gcd}(k + r, l + s),\\
		g_{4} & = \mathrm{gcd}(k - r, l - s).\\
	\end{split}\]
\end{proposition}

We record the following useful formula for the composition of polynomials $Y_k \circ P_n.$
\begin{lemma}\thlabel{Y_comp}
For any $k, l \in \N$, $Y_k \circ P_l = Y_{kl}$.
\end{lemma}

\begin{proof}
Recall from Equation $\eqref{eq:P composition}$, we have $P_k\circ P_l = P_{kl}.$ Thus, using $Y_k=P_k-1$ we obtain:
\[Y_k(P_l(x)) = P_k(P_l(x)) - 1 = P_{kl}(x) - 1 = Y_{kl}(x).  \]
\end{proof}

\begin{lemma}\label{lem:eta alg map}
Let $q^2$ be a primitive $n$th root of unity and let $t = q^{n^2}$. For any algebra homomorphism $\eta: \mathcal{S}_t^Y(\T^2) \to \mathcal{S}_q^Y(\T^2)$ that sends each torus knot $(k,l)$ to $(k,l) ^{[P_n]}$, we must have:
\[\eta(w) = \alpha_n w\]
where $\alpha_n = \frac{2((-1)^{n+1} + 1)}{(q + q^{-1})^2}.$
\end{lemma}

\begin{proof}
Recall that the product of $K,L \in \mathcal{S}_q^Y(\T^2)$, denoted $KL$, is obtained by drawing $K$ on top of $L$. Note that $t^2 = q^{2n^2} = 1$. Thus, by expanding $[2]_{t}h \in \mathcal{S}_t^Y$ with the skein relations, we get: 
\begin{align*}
[2]_{t}h = (1,0)(0,1) - t^2(-1,1) - (t^{-2}-1)(1,1) = (1,0)(0,1) - (-1,1),
\end{align*} which is written diagrammatically as:
\begin{figure}[h!]
\centering
\resizebox{14cm}{!}{

\tikzset{every picture/.style={line width=0.75pt}} 

\begin{tikzpicture}[x=0.75pt,y=0.75pt,yscale=-1,xscale=1]

\draw   (51,91) -- (151,91) -- (151,191) -- (51,191) -- cycle ;
\draw   (193,91) -- (293,91) -- (293,191) -- (193,191) -- cycle ;
\draw   (341,90) -- (441,90) -- (441,190) -- (341,190) -- cycle ;
\draw    (51,141) -- (151,141) ;
\draw    (101,91) -- (80,141) ;
\draw    (120,141) -- (100,191) ;
\draw    (193,141) -- (293,141) ;
\draw  [color={rgb, 255:red, 255; green, 255; blue, 255 }  ,draw opacity=1 ][fill={rgb, 255:red, 255; green, 255; blue, 255 }  ,fill opacity=1 ] (240,138) -- (246,138) -- (246,145.5) -- (240,145.5) -- cycle ;
\draw    (243,90.5) -- (243,141) -- (243,190.5) ;
\draw  [draw opacity=0] (341,137.5) .. controls (341,137.5) and (341,137.5) .. (341,137.5) .. controls (369.99,137.5) and (393.5,161.01) .. (393.5,190) -- (341,190) -- cycle ; \draw   (341,137.5) .. controls (341,137.5) and (341,137.5) .. (341,137.5) .. controls (369.99,137.5) and (393.5,161.01) .. (393.5,190) ;  
\draw  [draw opacity=0] (440.96,142.5) .. controls (440.96,142.5) and (440.96,142.5) .. (440.96,142.5) .. controls (411.97,142.48) and (388.48,118.96) .. (388.5,89.96) -- (441,90) -- cycle ; \draw   (440.96,142.5) .. controls (440.96,142.5) and (440.96,142.5) .. (440.96,142.5) .. controls (411.97,142.48) and (388.48,118.96) .. (388.5,89.96) ;  
\draw   (630.02,89.98) -- (629.98,189.98) -- (529.98,189.95) -- (530.02,89.95) -- cycle ;
\draw  [draw opacity=0] (582.52,89.96) .. controls (582.52,89.96) and (582.52,89.96) .. (582.52,89.96) .. controls (582.52,89.96) and (582.52,89.96) .. (582.52,89.96) .. controls (582.51,118.96) and (558.99,142.46) .. (530,142.45) -- (530.02,89.95) -- cycle ; \draw   (582.52,89.96) .. controls (582.52,89.96) and (582.52,89.96) .. (582.52,89.96) .. controls (582.52,89.96) and (582.52,89.96) .. (582.52,89.96) .. controls (582.51,118.96) and (558.99,142.46) .. (530,142.45) ;  
\draw  [draw opacity=0] (577.48,189.93) .. controls (577.48,189.93) and (577.48,189.93) .. (577.48,189.93) .. controls (577.48,189.93) and (577.48,189.93) .. (577.48,189.93) .. controls (577.51,160.93) and (601.04,137.45) .. (630.04,137.48) -- (629.98,189.98) -- cycle ; \draw   (577.48,189.93) .. controls (577.48,189.93) and (577.48,189.93) .. (577.48,189.93) .. controls (577.48,189.93) and (577.48,189.93) .. (577.48,189.93) .. controls (577.51,160.93) and (601.04,137.45) .. (630.04,137.48) ;  

\draw (25,130.4) node [anchor=north west][inner sep=0.75pt]    {$[2]_t$};
\draw (164,137.4) node [anchor=north west][inner sep=0.75pt]    {$=$};
\draw (303,130.4) node [anchor=north west][inner sep=0.75pt]    {$-t^{2}$};
\draw (447,130.4) node [anchor=north west][inner sep=0.75pt]    {$-\left( t^{-2} -1\right)$};
\draw (97,203.4) node [anchor=north west][inner sep=0.75pt]    {$h$};
\draw (206,203.4) node [anchor=north west][inner sep=0.75pt]    {$( 1,0)( 0,1)$};
\draw (366,203.4) node [anchor=north west][inner sep=0.75pt]    {$( -1,1)$};
\draw (561,203.4) node [anchor=north west][inner sep=0.75pt]    {$( 1,1)$};

\end{tikzpicture}

}
\end{figure}
\bigskip 

 Recall that $\eta((k,l)) = (k,l)^{[P_n]} = (k,l)^{[Y_n + 1]}$ and that $q^2$ is a primitive $n$th root of unity. Thus we obtain the following formulae for $\eta([2]_t h)$:
\begin{align*}
\eta([2]_{t}h) =& \eta((1,0)(0,1) - (-1,1))\\ =& (1,0)^{[Y_n + 1]}(0,1)^{[Y_n + 1]} - (-1,1)^{[Y_n + 1]}\\ 
=& ((n,0)_Y+\varnothing)((0,n)_Y + \varnothing) - ((-n,n)_Y + \varnothing) \\
=& (n,0)_Y(0,n)_Y + (n,0)_Y + (0,n)_Y + \varnothing -(-n,n)_Y - \varnothing \\
\overset{(\text{Prop } \ref{thm:multiplication table})}{=} & q^{-2n^2}(n,n)_Y + q^{2n^2}(n,-n)_Y + \frac{(-1)^{n+1}q^{2n^2} + 2 + (-1)^{n+1}q^{-2n^2}}{[2]_{q}^2} w \\
& + (n,0)_Y + (0,n)_Y - (-n,n)_Y \\
=& (n,n)_Y + \frac{2((-1)^{n+1} + 1)}{[2]_{q}^2}w +(n,0)_Y + (0,n)_Y.
\end{align*}

With the formula for $\eta([2]_{t}h)$ computed, we now compute $\eta(w)$:
\begin{align*}
\eta(w) =& \eta([2]_{t}h) - (0,1)^{[Y_n + 1]} - (1,0)^{[Y_n + 1]} - (1,1)^{[Y_n + 1]} + 3\varnothing\\
=& (n,n)_Y + \frac{2((-1)^{n+1} + 1)}{[2]_q^2}w +(n,0)_Y + (0,n)_Y \\ &- (0,n)_Y - \varnothing - (n,0)_Y - \varnothing -(n,n)_Y - \varnothing + 3\varnothing \\=&\frac{2((-1)^{n+1} + 1)}{[2]_{q}^2}w = \alpha_n w.
\end{align*}
\end{proof}

We record that \begin{align*}
\alpha_n = \begin{dcases}
0, & n \text{ even} \\ 
\frac{4}{[2]_{q}^2}, & n \text{ odd}.
\end{dcases}
\end{align*}

\begin{theorem}\label{thm: Frobenius torus}
Let $q^2$ be a primitive $n$th root of unity and let $t = q^{n^2}$. Then, there exists a unique algebra homomorphism $Fr^Y: \mathcal{S}_{t}^Y(\T^2) \rightarrow \mathcal{S}_{q}^Y(\T^2)$ that threads any multicurve by $P_n$.
\end{theorem}

\begin{proof}
Define a module homomorphism $Fr^Y: \mathcal{S}_t^Y(\T^2) \to S_{q}^Y(\T^2)$ by the following assignments on the basis of $\mathcal{S}_t^Y(\T^2)$ from Proposition \ref{thm:torus basis}:

\begin{align*}
Fr^Y((mk, ml)_Y) &= (nmk,nml)_Y \text{ for $(k,l)$ coprime and $m \in \Z_{>0}$}\\
Fr^Y(w) &= \alpha_n w\\
Fr^Y(\varnothing)&=\varnothing.
\end{align*}

We first remark that this definition of $Fr^Y$ on $w$ is forced by Lemma \ref{lem:eta alg map}.
Furthermore, the definition of $Fr^Y((mk,ml)_Y)$ is forced by \thref{Y_comp} along with the condition that we set $Fr^Y((k,l))=(k,l)^{[P_n]}$ for $k,l$ coprime since:

\begin{equation*}
Fr^Y((mk, ml)_Y) = Fr^Y((k, l)^{[Y_m]}) = (k,l)^{[Y_m\circ P_n]}= (k,l)^{[Y_{mn}]} = (mnk,mnl)_Y.    
\end{equation*}

Therefore, the module homomorphism $Fr^Y$ will be the unique algebra homomorphism satisfying our hypothesis once it is shown to be an algebra homomorphism.

To show $Fr^Y$ is an algebra homomorphism, it suffices to show that the map is multiplicative on pairs of basis elements of $\mathcal{S}_{t}^Y(\T^2)$. We prove this for the basis of $\mathcal{S}_{t}^Y(\T^2)$ set out in Proposition \ref{thm:torus basis} via the multiplication table from Proposition \ref{thm:multiplication table}. Since the map is automatically multiplicative on a pair of basis elements if one of the elements is $\varnothing,$ we must check that the following identities hold in $\Y(\T^2)$:

\begin{itemize}
\item[(i)] $Fr^Y(w^2) = (Fr^Y(w))^2$,
\item[(ii)] $Fr^Y(w(k,l)_Y) = Fr^Y(w)Fr^Y( (k,l)_Y)$ (and its opposite product),
\item[(iii)] $Fr^Y((k,l)_Y (r,s)_Y) = Fr^Y((k,l)_Y)Fr^Y( (r,s)_Y)$.
\end{itemize}

We will first check (i). From Proposition \ref{thm:multiplication table} we will use the identity $w^2 = 2[2]^2w$.
First, suppose that $n$ is even. In this case, we have that $\alpha_n = 0$, and we check that
\begin{align*}
Fr^Y(w^2) &= Fr^Y(2[2]_{t}^2 w) = 2[2]_{t}^2 \alpha_n w = 0 \\
(Fr^Y(w))^2 &= (\alpha_nw)^2 = 0,
\end{align*}
as desired.

Now, suppose that $n$ is odd. We first note that $[2]_{t}^2 = (q^{n^2} + q^{-n^2})^2 = 4$, and we check that:
\begin{align*}
Fr^Y(w^2) &= Fr^Y(2[2]_{t}^2 w) = 8 \alpha_nw = \frac{32}{[2]_{q}^2}w\\
(Fr^Y(w))^2 &= (\alpha_nw)^2 = \frac{16}{[2]_{q}^4} 2[2]_{q}^2w = \frac{32}{[2]_{q}^2}w,
\end{align*}

as desired.

\bigskip
We will next check the identity (ii). Its opposite will follow by symmetry. We will use the following identity from Proposition \ref{thm:multiplication table}:
\begin{align*}
&w(k,l)_Y = (k,l)_Yw = 2(-1)^{\gcd(k,l)}w.
\end{align*}
Then, we have:
\begin{align*}
&Fr^Y(w)Fr^Y((k,l)_Y) = \alpha_nw (nk, nl)_Y = \alpha_n 2(-1)^{n\gcd(k,l)} w , \\
&Fr^Y(w(k,l)_Y) = Fr^Y(2(-1)^{\gcd(k,l)}w) = \alpha_n2(-1)^{\gcd(k,l)}  w,
\end{align*}
For $n$ even, $\alpha_n = 0$ and equality holds. For $n$ odd, $(-1)^{n\gcd(k,l)} = (-1)^{\gcd(k,l)}$ and equality holds.

Lastly, we will check the identity (iii). We will use the identity from Proposition \ref{thm:multiplication table} that for $(k, l), (r, s) \in \Z^2 \setminus \{(0,0)\}$ we have
\begin{align*}
(k,l)_Y(r,s)_Y = q^{2d_0}(k+r, l+s)_Y + q^{-2d_0}(k-r,l-s)_Y + \alpha w,
\end{align*}

where we recall that
\begin{align*}
&d_0 = -\begin{vmatrix}
    k & l \\
    r & s
\end{vmatrix}, \\
&\alpha = \frac{(-1)^{g_1+g_2}2 + (-1)^{g_3+1}q^{2d_0} + (-1)^{g_4+1}q^{-2d_0}}{[2]_q^2}\\
&g_1 = \gcd(k, l),\\
&g_2 = \gcd(r, s),\\
&g_3 = \gcd(k+r, l+s), \\
&g_4 = \gcd(k-r, l-s). \\
\end{align*}

Let $k, l, r, s \in \Z$ satisfy $\gcd(k, l)=\gcd(r, s)=1.$ Let $i, j \in \Z_{>0}$. To prove (iii), it suffices to show
\begin{align}\label{eq:prime-coprime-torus}
    Fr^Y((ik,il)_Y (jr,js)_Y) = Fr^Y((ik,il)_Y)Fr^Y( (jr,js)_Y).
\end{align}
Computing the right-hand side of \eqref{eq:prime-coprime-torus} yields:

\begin{align*}
Fr^Y((ik,il)_Y)Fr^Y((jr,js)_Y) &=(nik,nil)_Y(njr,njs)_Y \\
&= q^{2d_n}(nik+njr,nil+njs)_Y \\&\qquad\qquad+ q^{-2d_n}(nik-njr, nil -njs)_Y + \alpha w.
\end{align*}
Note that $d_n = -\begin{vmatrix}
nik & nil \\
njr & njs
\end{vmatrix}$ is divisible by $n$, so $q^{2d_n} = 1$. Thus, this reduces to
\begin{equation}\label{eq:ch_1}
\begin{split}
Fr^Y((ik,il)_Y)Fr^Y((jr,js)_Y) &= (nik+njr,nil+njs)_Y \\ &\qquad+ (nik-njr, nil -njs)_Y + \alpha w
\end{split}
\end{equation}

where the formula for $\alpha$ specializes to:
\begin{align*}
\alpha &= \frac{(-1)^{n(i+j)}2 + (-1)^{ng_3' + 1} + (-1)^{ng_4' + 1}}{[2]_{q}^2},\\
\text{where }g_3' &\coloneq \gcd(ik + jr, il + js),\\
g_4' &\coloneq \gcd(ik - jr, il - js).
\end{align*}

\bigskip
\noindent Checking the left-hand side of \eqref{eq:prime-coprime-torus}, we have:
\begin{align*}
Fr^Y((ik, il)_{Y} (jr, js)_{Y}) &= 
Fr^Y(t^{2ijd_0}(ik+jr, il+js)_Y + t^{-2ijd_0}(ik-jr,il-js)_Y + \beta w)  \\
&\overset{(t^2 = 1)}{=}Fr^Y((ik+jr, il+js)_Y) + Fr^Y((ik-jr,il-js)_Y) + Fr^Y(\beta w) \\ 
&= (nik+njr,nil+njs)_Y + (nik-njr, nil -njs)_Y + \beta \alpha_n w
\end{align*}
where
\[\beta = \frac{(-1)^{i+j}2 +(-1)^{g_3'+1}t^{2ijd_0} 
+ (-1)^{g_4'+1}t^{-2ijd_0}}{[2]_t^2} = \frac{(-1)^{i+j}2 +(-1)^{g_3'+1}
+ (-1)^{g_4'+1}}{4}. \]

\bigskip
\noindent When $n$ is even we have $\alpha_n = 0$ and
\begin{align*}
\alpha = \frac{(-1)^{n(i+j)}2 + (-1)^{ng_3' + 1} + (-1)^{ng_4' + 1}}{[2]_{q}^2} = 0
\end{align*}
so that
\begin{align*}
Fr^Y((ik,il)_Y)Fr^Y((jr,js)_Y) &= (nik+njr,nil+njs)_Y + (nik-njr, nil -njs)_Y \\
&=
Fr^Y((ik, il)_{Y} (jr, js)_{Y}).
\end{align*}
When $n$ is odd, $\alpha_n = \frac{4}{[2]_q^2}$ and
\begin{align*}
\alpha &= \frac{(-1)^{n(i+j)}2 + (-1)^{ng_3' + 1} + (-1)^{ng_4' + 1}}{[2]_{q}^2} \\
&= \frac{(-1)^{(i+j)}2 + (-1)^{g_3' + 1} + (-1)^{g_4' + 1}}{[2]_{q}^2} = \frac{4}{[2]_{q}^2}\beta = \alpha_n \beta
\end{align*}
so that
\begin{align*}
Fr^Y((ik,il)_Y)Fr^Y((jr,js)_Y) &= (nik+njr,nil+njs)_Y + (nik-njr, nil -njs)_Y + \alpha_n\beta w \\
&=
Fr^Y((ik, il)_{Y} (jr, js)_{Y}).
\end{align*}
\end{proof}

The computation of $Fr^Y$ on the basis for $\mathcal{S}_t^Y(\T^2)$ yields the following corollary. 
\begin{corollary}
For the torus, $Fr^Y$ is not injective if $n$ is even. However, when $n$ is odd, $Fr^Y$ is injective.
\end{corollary}

\section{Frobenius homomorphism for the pair of pants}\label{Sec: Frobenius pants}
In this section we compute the image of the $\rm{PGL}_2$ Frobenius homomorphism for the pair of pants $Fr^Y:\mathcal{S}^Y_t(\Sigma_{0,3})\rightarrow \mathcal{S}^Y_q(\Sigma_{0,3})$ when $q^2$ is a primitive $n$th root of unity and $t=q^{n^2}.$ We will use the Chebyshev-Frobenius homomorphism of Bonahon-Wong \cite{Bonahon_2015} from the Kauffman bracket case, $Fr:\mathcal{S}_t(\Sigma_{0,3})\rightarrow \mathcal{S}_q(\Sigma_{0,3}).$

Recall that $\mathcal{S}_q(\Sigma_{0,3})$ is generated by the
elements

\begin{equation*}
\begin{tikzpicture}[x=0.05cm,y=0.05cm]
  \draw (-100,0) circle[radius=0.5cm];
  \fill[black] (-100,0) circle[radius=1.5pt];
  \fill[black] (-80,0)  circle[radius=1.5pt];
  \node at (-90,-20) {$x$};

  \draw (10,0) circle[radius=0.5cm];
  \fill[black] (-10,0) circle[radius=1.5pt];
  \fill[black] (10,0)  circle[radius=1.5pt];
  \node at (0,-20) {$y$};

  \draw (90,0) ellipse[x radius=1cm,y radius=0.5cm];
  \fill[black] (80,0)  circle[radius=1.5pt];
  \fill[black] (100,0) circle[radius=1.5pt];
  \node at (90,-20) {$z$};
\end{tikzpicture}
\end{equation*}

\noindent and $\mathcal{S}_q^Y(\Sigma_{0,3})$ is generated by

\begin{equation*}
\begin{tikzpicture}[x=0.05cm,y=0.05cm]
  \draw (-100,0) circle[radius=0.5cm];
  \fill[black] (-100,0) circle[radius=1.5pt];
  \fill[black] (-80,0)  circle[radius=1.5pt];
  \node at (-90,-20) {$a$};

  \draw (10,0) circle[radius=0.5cm];
  \fill[black] (-10,0) circle[radius=1.5pt];
  \fill[black] (10,0)  circle[radius=1.5pt];
  \node at (0,-20) {$b$};

  \draw (90,0) ellipse[x radius=1cm,y radius=0.5cm];
  \fill[black] (80,0)  circle[radius=1.5pt];
  \fill[black] (100,0) circle[radius=1.5pt];
  \node at (90,-20) {$c$};

  \draw (180,0) ellipse[x radius=1cm,y radius=0.5cm];
  \draw (180,10) -- (180,-10);
  \fill[black] (170,0) circle[radius=1.5pt];
  \fill[black] (190,0) circle[radius=1.5pt];
  \node at (180,-20) {$p$};
\end{tikzpicture}
\end{equation*}

The set
$\{x^ry^sz^t \mid r, s, t \in \mathbb{N}_0\}$ constitutes a monomial
basis of $\mathcal{S}_q(\Sigma_{0,3})$, and the set
$\{a^rb^sc^tp^u \mid r, s, t \in \mathbb{N}_0, u = 0, 1\}$ constitutes a
monomial basis of $\Y(\Sigma_{0,3})$. In order to have cleaner formulas at the end of this section, we make the change of generator $d = p - \frac1{[2]}(a + b + c + 1) \in \mathcal{S}_q^Y(\Sigma_{0,3})$, noting that $\{a, b, c, d\}$ also generate $\mathcal{S}_q^Y(\Sigma_{0,3})$.

Recall the map $\varphi$ from \eqref{eq:phimap}. We will construct $Fr^Y$ satisfying the following commutative diagram.

\begin{equation}\label{diag: Frobenius pants}
\begin{tikzcd}[row sep=large, column sep=large]
  \mathcal{S}_t(\Sigma_{0,3})
    \arrow[r, "\mathrm{Fr}"]
  & \mathcal{S}_q(\Sigma_{0,3})
  \\
  \mathcal{S}_t^Y(\Sigma_{0,3})
    \arrow[u, "\varphi_t"]
    \arrow[r, "\mathrm{Fr}^Y"']
  & \mathcal{S}_q^Y(\Sigma_{0,3})
    \arrow[u, "\varphi_q"']
\end{tikzcd}
\end{equation}

\begin{theorem}\label{Thm:Frobenius pants}
When $q^2$ is a root of unity of order $n$ and $t=q^{n^2},$ there exists a unique algebra homomorphism $Fr^Y: \mathcal{S}_t^Y(\Sigma_{0,3}) \rightarrow \mathcal{S}_q^Y(\Sigma_{0,3})$ making the diagram in (\ref{diag: Frobenius pants}) commute.

In particular, the value of $Fr^Y$ on our generators is given by

\begin{equation*}
Fr^Y(l)=P_n(l) \text{ for } l \in \{a,b,c\},
\end{equation*}
where $P_n(x)$ is the unique polynomial satisfying $P_n(x^2-1)=T_n(x)^2-1$ from Definition \ref{defn:Pn-family}, and
\begin{equation*}
Fr^Y(d)=\begin{cases}
    d\cdot O_n(a)O_n(b)O_n(c) \text{ if $n$ is odd},\\
    E_n(a)E_n(b)E_n(c) \text{ if $n$ is even.}
\end{cases}
\end{equation*}

Here, $O_n(x)$ is defined to be the unique polynomial satisfying $O_n(x^2-1)=\frac{T_n(x)}{x},$ and $E_n(x)$ is defined to be the unique polynomial satisfying $E_n(x^2-1)=T_n(x).$
\end{theorem}

\begin{proof}
We have that the map $\varphi_q:\mathcal{S}_q^Y(\Sigma_{0,3}) \rightarrow \mathcal{S}_q(\Sigma_{0,3})$ is injective on the three punctured sphere by arguments of \cite{Chb12}. Thus, if we can construct elements $Fr^Y(a),Fr^Y(b), Fr^Y(c)$, and $Fr^Y(d)$ in $ \Y(\Sigma_{0,3})$ satisfying the commutative diagram (\ref{diag: Frobenius pants}), then we will have produced a construction of the desired unique algebra homomorphism $Fr^Y$.

Consider a loop $l \in \{a,b,c\}$. We know that $\varphi_q(l),\varphi_t(l)=l^2-1.$ We compute that

\begin{equation*}
Fr\circ \varphi_t(l)=Fr(l^2-1)=T_n(l)^2-1.
\end{equation*}

On the other hand,

\begin{equation*}
\varphi_q(P_n(l))=P_n(l^2-1)=T_n(l)^2-1.
\end{equation*}

Thus, we can make the diagram commute on the generator $l$ by setting $Fr^Y(l)=P_n(l).$

Next, we will construct the element $Fr^Y(d) \in \Y(\Sigma_{0,3})$ making the diagram commute.

Recall that for $v \in \{t,q\},$ we have $d = p - \frac1{[2]_v}(a + b + c + 1) \in \mathcal{S}_v^Y(\Sigma_{0,3}).$ We first compute that

\begin{equation}
    \varphi_v(p) =\quad \tikz[baseline=0.1ex]{\draw (-0.5,-0.5) rectangle (0.5,0.5); \draw (-2.5,-0.5) rectangle (-1.5,0.5); \draw (1.5,-0.5) rectangle (2.5,0.5); \node (0,0) {2};  \node at (-2, 0) {2}; \node at (2, 0) {2}; \draw (0.25,0.5) .. controls (0.75,1) and (1.25,1) .. (1.75,0.5); \draw (-0.25,0.5) .. controls (-0.75,1) and (-1.25,1) .. (-1.75,0.5); \draw (0.25,-0.5) .. controls (0.75,-1) and (1.25,-1) .. (1.75,-0.5); \draw (-0.25,-0.5) .. controls (-0.75,-1) and (-1.25,-1) .. (-1.75,-0.5); \draw (-2.25,-0.5) .. controls (-0.91,-2) and (0.91,-2) .. (2.25,-0.5); \draw (-2.25,0.5) .. controls (-0.91,2) and (0.91,2) .. (2.25,0.5); \fill[black] (-1, 0) circle (2pt); \fill[black] (1,0) circle (2pt)} \quad = xyz + \frac1{[2]_v}(x^2 + y^2 + z^2 - 2) \in \mathcal{S}_v(\Sigma_{0,3}).
\end{equation}

From this, we deduce the simple formula

\begin{equation*}
\varphi_v(d)=xyz.    
\end{equation*}

Consequently,

\begin{equation*}
Fr \circ \varphi_t(d)=T_n(x)T_n(y)T_n(z).
\end{equation*}

Now, suppose that $n$ is odd. Then $T_n$ is an odd
$n$\raise0.5ex\hbox{\scriptsize th} degree polynomial. Hence define
$\tilde{T_n} := \frac{T_n(x)}{x}$, and note that $\tilde{T_n}$ is an even
polynomial. Since any even polynomial can be written as a polynomial in $x^2 -1$, we set
$O_n$ as the polynomial satisfying $O_n(x^2-1) = \tilde{T_n}(x)$. Then note

\begin{align*}
Fr \circ \varphi_t(d)&=T_n(x)T_n(y)T_n(z)\\
&=xyz\tilde{T}_n(x)\tilde{T}_n(y)\tilde{T}_n(z)\\
&=xyzO_n(x^2-1)O_n(y^2-1)O_n(z^2-1).
\end{align*}

On the other hand,

\begin{equation*}
\varphi_q(dO_n(a)O_n(b)O_n(c))=xyzO_n(x^2-1)O_n(y^2-1)O_n(z^2-1).
\end{equation*}

Therefore, when $n$ is odd, setting $Fr^Y(d)=dO_n(a)O_n(b)O_n(c)$ makes (\ref{diag: Frobenius pants}) commute.

Now, suppose that $n$ is even. Then $T_n$ is an even polynomial, and must be a polynomial in $x^2-1.$ We then set $E_n(x)$ to be the unique polynomial satisfying $E_n(x^2-1)=T_n(x).$

Then 

\begin{equation*}
Fr \circ \varphi_t(d)=Fr(xyz)=T_n(x)T_n(y)T_n(z)=E_n(x^2-1)E_n(y^2-1)E_n(z^2-1).
\end{equation*}

On the other hand,

\begin{equation*}
\varphi_q(E_n(a)E_n(b)E_n(c))=E_n(x^2-1)E_n(y^2-1)E_n(z^2-1).
\end{equation*}

Thus, if $n$ is even, setting $Fr^Y(d)=E_n(a)E_n(b)E_n(c)$ makes (\ref{diag: Frobenius pants}) commute.
\end{proof}

\begin{remark}
One may see that $Fr^Y$ for the pair of pants is injective either by using the monomial basis directly or by using the fact that both $\varphi_t$ and $Fr$ are known to be injective.
\end{remark}
\begin{remark}
Here, we collect some facts about the polynomials $E_n$ and $O_n$ appearing in Theorem \ref{Thm:Frobenius pants}. The polynomials satisfy the following recursive identities.
\begin{align*}
    O_1(x) &= 1, &E_0 &= 2,\\
    O_3(x) &= x - 2, &E_2 &= x - 1,\\
    O_5(x) &= x^2 - 3x + 1, &E_4 &= x^2 - 2x - 1,
\end{align*}
and in general
$$O_{2n+1} = E_{2n} - O_{2n-1}, \qquad E_{2n+2} = xO_{2n+1} - O_{2n-1},$$
which can be seen by applying the recurrence relation $T_n = xT_{n-1} - T_{n-2}$.

It is curious to see that the polynomials $O_n$ also satisfy the same recursion as the polynomials $P_n$. Namely,
\begin{equation}\label{evilcheby-recurrence}
        O_{2n+1}(x) = xO_{2n-1}(x) - xO_{2n-3}(x) + O_{2n-5}(x).
\end{equation}
Indeed, note that \eqref{evilcheby-recurrence} is true if and only if
\begin{equation*}
     \tilde{T}_{2n+1}(x) = (x^2-1)\tilde{T}_{2n-1}(x) - (x^2-1)\tilde{T}_{2n-3}(x) + \tilde{T}_{2n-5}(x),
\end{equation*}
or, equivalently,
\begin{equation}\label{eq:evil-Cheby-to-normalCheby}
    T_{2n+1}(x) = (x^2 - 1)T_{2n-1}(x) - (x^2 - 1)T_{2n-3}(x) + T_{2n-5}(x).
\end{equation}
Now recall that $T_n$ is the unique polynomial satisfying $T_n(q + q^{-1}) = q^n + q^{-n}$ for every $q \in \mathbb{C}\setminus\{0\}$. Hence, it suffices to verify \eqref{eq:evil-Cheby-to-normalCheby} when $x = q + q^{-1}$. Indeed,
\begin{multline*}
    q^{2n+1} + q^{-2n-1} = (q^2 + 1 + q^{-2})(q^{2n-1} + q^{-2n + 1}) \\- (q^2 + 1 + q^{-2})(q^{2n-3} + q^{-2n + 3}) + (q^{2n-5} + q^{-2n+5}),
\end{multline*}
so we conclude that \eqref{evilcheby-recurrence} holds.
\end{remark}


\subsection{Acknowledgments}

This work was supported by National Science Foundation DMS-2136090 and was completed during the Geometry \& Topology REU at UCLA during Summer 2025. We thank Ko Honda for overseeing the program, and we thank Roman Krutowski, Jung Joo Suh, and Yan Tao for their mentorship and support.

\subsection{Tool use}
The authors acknowledge the use of ChatGPT for help creating figures, for translating 
Asymptote figures into Tikz figures, and for finding BibTex code containing bibliographical information for some of the works cited. All figures and references have been verified by the authors, who take sole responsibility for the contents of the manuscript.

\bibliographystyle{amsalpha}
\bibliography{references.bib}

\end{document}